\documentclass[11pt,a4paper]{article}
\usepackage{amsmath,amsthm,amssymb}
\usepackage{graphicx}
\usepackage{subfigure}
\usepackage{multirow}
\usepackage{epstopdf}
\usepackage{verbatim}
\usepackage{color}
\usepackage{pgf}

\usepackage{tikz}
\usetikzlibrary{matrix}

\newtheorem{theorem}{Theorem}[section]
\newtheorem{lemma}{Lemma}[section]

\newtheorem{remark}{Remark}[section]

\numberwithin{equation}{section}

\title{Strong approximation of the time-fractional Cahn--Hilliard equation driven by a fractionally integrated additive noise}
\author{Mariam Al-Maskari\footnotemark[2]
\and Samir Karaa\footnotemark[2]}
\newcommand{\cop}{{^C}\partial_t^{\alpha}}
\newcommand{\op}{\partial_t^{\alpha}}

\newcommand{\dop}{\partial_\tau^{\alpha}}

\begin{document}
\date{}
\maketitle

\maketitle

\renewcommand{\thefootnote}{\fnsymbol{footnote}}
\footnotetext[2]{Department of Mathematics, Sultan Qaboos University,
 Al-Khod 123, Muscat, Oman (m.almaskari@squ.edu.om, skaraa@squ.edu.om.
 This research is supported by Sultan Qaboos University grant IG/SCI/MATH/24/02}
%
  
\renewcommand{\thefootnote}{\arabic{footnote}}

\begin{abstract}
In this paper, we consider the numerical approximation of
 a time-fractional stochastic Cahn--Hilliard equation driven by an additive fractionally integrated Gaussian noise. The model involves a Caputo fractional derivative in time of order $\alpha\in(0,1)$  and a fractional time-integral noise of order $\gamma\in[0,1]$. The numerical scheme approximates  the model  by a piecewise linear finite element method in space and  a convolution quadrature in time (for both time-fractional operators), along with the $L^2$-projection for the noise.
 We carefully investigate the spatially semidiscrete and fully discrete schemes, and obtain strong convergence rates by using clever energy arguments. The temporal H\"older continuity property of the solution played a key role in the error analysis.  Unlike the stochastic
Allen--Cahn equation, the presence of the unbounded elliptic operator in front of the cubic nonlinearity
in the underlying model adds complexity and challenges to the error analysis.
To overcome these difficulties, several new techniques and error estimates are developed.  The study concludes with numerical examples that validate the theoretical findings.
\end{abstract}



\section{Introduction} 
We consider the  time-fractional stochastic Cahn--Hilliard equation perturbed by 
a fractionally integrated noise,
\begin{subequations}\label{main}
\begin{alignat}{2}\label{main1}
\cop u -\Delta(-\Delta u + \phi(u))&=\partial_t^{-\gamma}\dot W(t) &&\quad\mbox{ in }\mathcal{D}\times (0,T ],
\\  \label{main2}
u(0)&=u_0 && \quad\mbox{ in }\Omega,
\\   \label{main3}
\frac{\partial u}{\partial \nu}=0, \; \frac{\partial }{\partial \nu}(-\Delta u+\phi(u))&=0 &&\quad\mbox{ on }\partial\mathcal{D}\times (0,T],
\end{alignat}
\end{subequations}
where $\mathcal{D}$ is a bounded convex domain in $\mathbb{R}^d$,  $d=1,2,3,$  with a polygonal boundary $\partial \mathcal{D}$,  $T >0$ is a fixed time,  and $\partial/\partial \nu$ denotes the outward normal derivative. The operator $\partial_t^{-\gamma}$ is the Riemann--Liouville fractional
integral of order $\gamma$ defined by
\begin{equation*} \label{Ba}
\partial_t^{-\gamma}\varphi(t):=\frac{1}{{\Gamma(\gamma)}}\int_0^t (t-s)^{\gamma-1}\varphi(s)\,ds
\end{equation*}
and $\cop:=\partial_t^{\alpha-1}\partial_t$ is the Caputo fractional derivative in time of order $\alpha$ ($0<\alpha<1$). 
In \eqref{main1}, the nonlinear function  $\phi(u):=u^3-u$, and the noise $\{W(t)\}_{t\geq 0} $ is a $Q$-Wiener process in $H:=L^2(\mathcal{D})$ 
 with respect to a filtered probability space $(\Omega, \mathcal{F}, \mathcal{P}, \{\mathcal{F}_t\}_{t\geq 0})$, and the initial data $u_0$ is an $\mathcal{F}_0$-measurable 
 random variable with values in $H$.

The non-randomized form of problem \eqref{main} is designed to depict phase separation and coarsening phenomena in a molten alloy, as well as spinodal decomposition for a binary mixture, as detailed in \cite{6,8,7}. Introducing noise to the physical model is a natural approach, serving either as an external random perturbation or as a means to address the lack of knowledge about certain involved physical parameters. For instance, in \cite{4,12}, and the references cited therein, the authors assert that only the stochastic version is capable of accurately describing the entire decomposition process in a binary alloy.

For the  deterministic Cahn--Hilliard equation  
\begin{equation}\label{eq1}
\cop u -\Delta(-\Delta u + \phi(u))=0,
\end{equation}
several numerical studies are available.
For the classical case $(\alpha=1)$, numerical studies with optimal error estimates  are available in the literature. In \cite{elliot-1989},  Elliott {\it et al.} derived optimal order bounds using a second-order splitting method assuming high regularity assumptions on initial data. In \cite{EL-1992}, Elliott and Larsson payed much attention to the solution regularity, and obtained error bounds for solutions with smooth and nonsmooth initial data, which are optimal with respect to the polynomial degrees and the regularity of initial data. 

The time-fractional model \eqref{eq1}   with $\alpha\in (0,1)$  
was investigated  by a few authors. 
 In \cite{Liu-2018}, Liu {\it et al.} proposed a numerical scheme using  an $L1$-approximation in time and a Fourier spectral method in space. In \cite{Tang-2019},  Tang  {\it et al.} showed that the  stabilized $L1$-scheme  admits an energy dissipation law of an integral type.
 In  the recent work \cite{ZZW-2020}, Zhang  {\it et al.} investigated a numerical scheme based on an  $L1^+$ formula in time and a second-order convex-splitting technique that treats the nonlinear term semi-implicitly.
 In all these studies, the existence and regularity of the exact solution were not taken into consideration.
For  Cahn--Hilliard equations with fractional order operators in space, we refer to \cite{MM-2017} and references therein.  Recently, in \cite{MK-2021}, Al-Maskari and Karaa established optimal error estimates   for the piecewise linear Galerkin finite element (FE) solution  with respect to the regularity of initial data. 
More precisely, they proved convergence rates of order $O(t^{-(2-\nu)/4}h^2)$ ($h$ denoting the maximum diameter of the spatial mesh elements) for $t \in (0, T]$  and initial data $u_0\in \dot H^\nu (\Omega),$  $\nu\in[1,2]$. Fully discrete schemes have also been analyzed and related convergence rates  were obtained for both smooth and nonsmooth initial data.

The numerical approximation of the classical stochastic  equation
\begin{equation}\label{eq2}
u_t -\Delta(-\Delta u + \phi(u))=\dot W(t),
\end{equation}
i.e., with $\alpha=1,\; \gamma=0$, was widely discussed in the literature; see,  for the instance,  \cite{9,10,11,12,13,14,15,16,17}. Notably, in \cite{9,10,11},  proofs of the strong convergence of mixed finite element methods for \eqref{eq2} are presented, though without a specified rate. The analysis in \cite{9,10,11} involves establishing a priori moment bounds with large exponents and in higher-order norms, and employing energy arguments.
In \cite{12}, the recovery of strong convergence rates for mixed finite element methods applied to the stochastic Cahn--Hilliard equation was achieved by utilizing a priori strong moment bounds of the numerical approximations, based on the existing results in \cite{9, 10, 11}. 
 In \cite{18}, an analysis of the strong convergence rate is presented for an implicit fully discretization with an unbounded noise diffusion in dimension one. Strong convergence rates for a fully discrete mixed FE method with a gradient-type multiplicative noise, where the noise process is a real-valued Wiener process, are derived in \cite{15}.
In \cite{19}, a tamed exponential Euler method in time combined with a spectral Galerkin method is considered. The resulting explicit scheme was found to be computationally efficient compared to other implicit schemes. Another explicit time-stepping scheme is investigated in \cite{20}, with strong convergence rates being obtained.

The numerical investigation of \eqref{eq2} with the time-derivative $u_t$ being substituted with the Caputo fractional derivative $\cop u$  is notably limited. To the best of our knowledge, this time-fractional model has only been considered in \cite{21}, where an analysis was conducted using a finite element discretization in space, coupled with a time approximation method that approximates the Mittag-Leffler function. Error estimates for both semidiscrete and fully discrete schemes were presented, without accompanying numerical experiments. 

In this paper, we investigate the strong convergence of the numerical solution of the model \eqref{main}. 
We consider a standard Galerkin linear FE method in space. We approximate the time-fractional derivative by a convolution quadrature and the noise by an $L^2$-projection. 
Given that the FE method incorporates the discrete operator $(A_h)^2$ instead of $(A^2)_h$, the analysis faces a significant challenge, as highlighted in \cite{elliot-1989}.
A delicate analysis is presented to effectively  overcome these difficulties and simplify the error analysis.
For the fully discrete problem, we conveniently exploit the properties of the associated discrete time evolution operators, as described in \cite{MK-2019,Karaa-2020}, which  simplifies the analysis of nonlinear schemes and allows one to achieve pointwise-in-time optimal error estimates.
New technical results (Lemmas \ref{error_integral} and \ref{sum_e_2}) are also found to play a key role in the analysis.

Throughout the paper,  $\alpha\in (0,1)$ and $\gamma\in [0,1]$. Our results require that $\|A^{(\beta-2)/2}Q^{1/2}\|_{HS}<\infty$ for $\beta\in[1,3]$, where $\|\cdot\|_{HS}$ denotes the Hilbert-Schmidt norm. The initial data  $u_0\in L^{16p}(\Omega;H_0^2(\mathcal{D}))$ for some $p\geq1$. 
Then, under  assumptions on the regularity of exact solution $u$, we prove a
 H\"older continuity property of $u$ in time  which turns out to be essential in proving convergence of numerical schemes. 
 
This work is organized as follows: In Section 2, we introduce the necessary notations, 
define the Winner process and derive preliminary estimates for the stochastic convolution. In Section 3, we recall some properties of the nonlinear source $\phi(u)$ and discuss the regularity in time of the exact solution $u$. The semidiscrete FE scheme is presented in Section 4,  and error estimates are derived.  Section 5 is devoted to the analysis of the fully discrete scheme  using a convolution quadrature in time.   Finally, Section 5 presents some numerical results to validate our theoretical estimates.

\section{Preliminaries} \label{sec:notation}
In this section, we introduce the  main notations and define the operators and functional spaces that will be used afterwards.
\subsection{Norms and operators}
Let $H$ be equipped with the usual inner product $(\cdot,\cdot)$ and norm $\|\cdot\|$,  and let $H^s:=H^s(\mathcal{D})$ denote the standard Sobolev space. Define  $\dot H:=\{v\in H,(v,1)=0\}$ and denote by $P:L^2\to \dot H$ the orthogonal projection given by 
$Pv=v-|{ \cal D}|^{-1}\int_{\cal D}v\,dx$.  Let $A=-\Delta$ be the negative Neumann Laplacian operator  with domain 
$$
{D}(A)=\left\{v\in H^2\cap \dot H: \frac{\partial v}{\partial \nu}=0 \mbox{ on } \partial \mathcal{D}\right\}.
$$
Then $A$ is  selfadjoint, positive definite, unbounded, linear on $\dot H$ with compact inverse. When extended to $H$ as $Av:=APv$, the linear operator $A$ has an orthonormal basis $\{\varphi_j\}_{j=0}^\infty$ of $H$ with corresponding eigenvalues 
$\{\lambda_j\}_{j=0}^\infty$ such that 
$$
0=\lambda_0<\lambda_1\leq \lambda_2\leq \cdots \leq\lambda_j \leq \cdots, \quad  \lambda_j \sim j^{\frac{2}{d}}\to \infty.
$$
Note that the first eigenfunction is a constant, i.e., $\varphi_0=|{\cal D}|^{\frac{1}{2}}$ and  $\{\varphi_j\}_{j=1}^\infty$ is an orthonormal basis of $\dot H$. By spectral method, we define the spaces $\dot H^s=D(A^{s/2})$, for $s \geq 0$, with norm  $\|v\|_{\dot H^s}=\|A^{s/2}v\|=\left(\sum_{j=1}^\infty \lambda_j^s(v,\varphi_j)^2\right)^{1/2}.$ For $s<0$, $ \dot H^s$ is the dual space of $ \dot H^{-s}$.   Note that for integer $s\geq 0$, $\dot H^s$ is a subspace of $H^s\cap \dot H$  characterized by some boundary conditions,  and that $\|\cdot\|_{\dot H^s}$ and $\|\cdot\|_{H^s}$ are equivalent on $\dot H^s$, see 
\cite{Mclean-2010} and \cite[Chapter 2]{Miranville-2019} for more details.
We have, for instance,  $\dot H^0 = \dot H$, $\dot H^1 = H^1\cap \dot H$ and $\dot H^2 = H^2\cap \dot H$. The norm $\|v\|_{\dot H^1}=\|A^{1/2}v\|=\|\nabla v\|$ is equivalent to $\|v\|_{H^1}$ on $\dot H^1$, and $\|v\|_{\dot H^2}=\|A v\|$ is equivalent to  $\|v\|_{H^2}$  on $\dot H^2$. Then, by interpolation, we have  in particular  $\|v\|_{H^s}\leq C\|v\|_{\dot H^s} \forall v\in \dot H^s$, $s\geq 0$.


\subsection{The Wiener process}
Let $L(H)$ be the space of bounded linear operators from $H$ into $H$. Let $Q\in L(H)$ be a selfadjoint, positive semidefinite operator.
We assume that the noise $W(t)$ is an $H$-valued $Q$-Wiener process given by an orthogonal expansion 
\begin{equation}\label{noise}
W(t)=\sum_{j=1}^\infty \gamma_j^{1/2}\beta_j(t)e_j,
\end{equation}
where $\{\gamma_j\}_{j=1}^\infty$ and $\{e_j\}_{j=1}^\infty$ are the eigenvalues and  an orthonormal basis of corresponding eigenfunctions of $Q$. Moreover, $\beta_j(t)$ are independent and identically distributed (i.i.d.) real-valued Brownian motions.
Thus, the regularity of $W(t)$ is observed from the decay rate of $\gamma_j\to0$; the faster the decay the smoother the noise.
For the mild solution of \eqref{main} to preserve mass, we assume that the average 
$(W(t),1)=0$ for all $t\geq 0$, see \cite{10,MK-2019}. That is, $W(t)$ is $\dot H$-valued, and thus, the covariance operator $Q$ belongs to $L(\dot H)$.

Now, denote $\mathcal{L}_2^0 $  the space of Hilbert-Schmidt operators $\psi$ from $Q^{1/2}(H)$ to $H$ with induced norm 
$$\|\psi\|_{\mathcal{L}_2^0 }=\left(\sum_{j=1}^\infty\|\psi Q^{1/2}e_j\|^2 \right)^{1/2}.$$  
Also, let $L^p(\Omega; \dot H^s)$ denote the space of $\dot H^s$-valued $p$-times integrable random variables with norm 
$$
\|v\|_{L^p(\Omega;\dot H^s)}=(\mathbb{E}\|v\|_{\dot H^s}^p)^{1/p}=\left(\int_\Omega \|v(\omega)\|_{\dot H^s}^pd\mathbb{P}(\omega)\right)^{1/p},
$$ 
 where $\mathbb{E} $ denotes the expected value. Finally, we recall Burkholder's inequality \cite[Lemma 7.2]{DZ-1992}, which will be used frequently hereafter. Let $\psi: [0,T]\times \Omega\rightarrow \mathcal{L}_2^0 $ be any predictable stochastic process that satisfies $\|\psi\|_{L^p(\Omega;L^2(0,T;\mathcal{L}_2^0 ))}<\infty $ for some $p\geq 2$, then we have 
 \begin{equation}\label{burk}
 \left\|\int_0^T \psi(s)dW(s)\right\|_{L^p(\Omega;H )}\leq C_p \|\psi\|_{L^p(\Omega;L^2(0,T;\mathcal{L}_2^0))}. 
 \end{equation}

\subsection{Representation of the Solution}

Formally, we can express problem \eqref{main} as follows  
\begin{equation}\label{2.5}
\cop u+A^2u+AP\phi(u)=\partial_t^{-\gamma}\dot W(t),\quad t>0,\quad u(0)=u_0\in \dot H.
\end{equation}
Given that  $A$ is  selfadjoint and positive definite on $\dot H$, for any $0<\theta < \pi$, the inequality 
\begin{equation}\label{res3}
\|A^\mu (z^\alpha I+A^2)^{-1}v\|\leq M |z|^{-\alpha(1-\mu/2)} \|v\| \quad \forall z\in \Sigma_\theta,\quad  \quad\ v\in \dot H, \mu\in[0,2]
\end{equation}
holds, where $M=M_\theta$  and $z\in \Sigma_{\theta}:=\{\,z\neq 0,\, |\arg z|< \theta\}$.  
Applying Duhamel's principle, the \textit{mild} solution to problem \eqref{main} is represented as 
\begin{equation}\label{form-1}
u(t)=E_{1-\alpha}(t)u_0-\int_0^t E_0(t-s)AP\phi(u(s))ds+\int_0^tE_{\gamma}(t-s)dW(s),\quad t>0.
\end{equation}
 Here, the operator $E_m(t)$, for $m\geq 0$, is defined by
\begin{equation}\label{ops}
 E_m(t) = \frac{1}{2\pi i}\int_{\Gamma_{\theta,\delta}}e^{zt} z^{-m}(z^\alpha I+A^2)^{-1} \,dz,
\end{equation}
where the contour $\Gamma_{\theta,\delta}=\{\rho e^{\pm i\theta}:\rho\geq \delta\}\cup\{\delta e^{i\psi}: |\psi|\leq \theta\}$, with $\theta\in (\pi/2,\pi)$  and $\delta> 0$,  
is oriented with an increasing imaginary part.
The smoothness of the  operator $E_m$ is provided  in the following lemma, the proof is outlined in \cite{MK-2021}. 
\begin{lemma}\label{eE}Let $u_0\in \dot{H}^\nu(\Omega)$. Then, for  $m\geq 1$ and $0\leq q,\nu\leq 4$, 
\begin{equation}\label{dE}
t^{\ell}\Vert \partial_t^\ell E_m(t)u_0\Vert_{\dot{H}^q} \leq c t^{-\alpha(q-\nu)/4+m+\alpha-1} \Vert u_0 \Vert_{\dot{H}^\nu},\quad\ \ell=0,1
\end{equation}
where for $\ell=0,\, 0\leq \nu\leq q\leq 4$ .  
\end{lemma}

\subsection{Stochastic convolution} \label{sec:SC}
Now, we introduce the  stochastic convolution 
$$
W_A(t)=\int_0^tE_{\gamma}(t-s)dW(s),\quad t>0,
$$
which appears on the right-hand side of \eqref{form-1}. Clearly, $W_A(t)$ serves as the mild solution of \eqref{main} with $\phi(u)=0$ and $u_0=0$.
By employing Burkholder's inequality \eqref{burk} and the bound \eqref{dE} with $m=\gamma$, $\ell=0$ and $q=3-\beta$, we obtain,  for $p\geq 2$,
\begin{eqnarray*}
\|W_A(t)\|_{L^{p}(\Omega;\dot H^1)}&= & \left\|\int_0^{t}A^{(3-\beta)/2}E_{\gamma}(t-s)A^{(\beta-2)/2} dW(s)\right\|_{L^p(\Omega;H)}\\
&\leq&C\left(\int_0^{t}(t-s)^{2[\alpha(1+\beta)/4+\gamma-1]}ds\right)^{1/2}\|A^{(\beta-2)/2}\|_{\mathcal{L}_2^0}\\
&\leq&ct^{\eta}\|A^{(\beta-2)/2}\|_{\mathcal{L}_2^0},
\end{eqnarray*}
where it is assumed that
\begin{equation}\label{eta}
\eta:=\alpha(1+\beta)/4+\gamma-1/2>0.
\end{equation}
This condition will be consistently enforced throughout the paper. 
Note  that $\eta <3/2$ holds when $\beta\in [1,3]$.
\begin{lemma}\label{LL} 
Let  $\beta\in [1,3]$  and  $\rho=\max\left\{\dfrac{q+2-\beta}{4},\; 0\right\} $, $q=0,1$.  Set 
\begin{equation}\label{theta}
\theta=\alpha (1- \rho)+\gamma.
\end{equation}
If $\theta>1/2$, then for $0<t_1<t_2$,
\begin{equation}\label{i2}
\|W_A(t_2)-W_A(t_1)\|_{L^{2p}(\Omega;\dot H^q)}\leq c(t_2-t_1)^{\min\{\theta-\frac{1}{2}-\epsilon_0,1\}}\|A^{(\beta-2)/2}\|_{\mathcal{L}_2^0},\quad p\geq 1,
\end{equation}
where $\epsilon_0 >0$ if $\theta=3/2$ and $\epsilon_0 =0$ otherwise.
\end{lemma}
\begin{proof}  
Assuming $t_1<t_2$, the difference $W_A(t_2)-W_A(t_1)$ is written as 
\begin{equation}\label{Wdiff}
W_A(t_2)-W_A(t_1) = \int_0^{t_1}(E_{\gamma}(t_2-s)-E_{\gamma}(t_1-s))dW(s) +\int_{t_1}^{t_2}E_{\gamma}(t_2-s)dW(s)=:I_1+I_2.
\end{equation}
For $1/2<\theta<3/2$, we follow the analysis of the proof given in \cite[Lemma A.2]{BYZ-2019}. Thus, we have  
\begin{eqnarray*}
\|I_1\|_{L^{2p}(\Omega;\dot H^q)} &= & \left\|\int_{t_1}^{t_2}\int_0^{t_1}A^{q/2}E_{\gamma}'(t-s)dW(s)dt \right\|_{L^{2p}(\Omega;H)}\\
  &\leq & \int_{t_1}^{t_2}\left\|\int_0^{t_1}A^{q/2}E_{\gamma}'(t-s)dW(s)\right\|_{L^2(\Omega;H)}dt \\
&\leq & \int_{t_1}^{t_2}\left(\int_0^{t_1}\|A^{(q+2-\beta)/2}E_{\gamma}'(t-s)\|^2 \|A^{(\beta-2)/2}\|_{\mathcal{L}_2^0}^{2}ds\right)^{1/2}dt \\
&\leq & c\int_{t_1}^{t_2}\left(\int_0^{t_1}(t-s)^{2(\theta-2)}ds \|A^{(\beta-2)/2}\|_{\mathcal{L}_2^0}^{2}\right)^{1/2}dt\\
 &\leq & c\int_{t_1}^{t_2}\left((t-t_1)^{2\theta-3} \|A^{(\beta-2)/2}\|_{\mathcal{L}_2^0}^2\right)^{1/2}dt\\
 &\leq & c(t_2-t_1)^{\theta-1/2} \|A^{(\beta-2)/2}\|_{\mathcal{L}_2^0}.
\end{eqnarray*}
For $\theta=3/2$, we get 
\begin{eqnarray*}
\|I_1\|_{L^{2p}(\Omega;\dot H^q)}^2
&\leq &c\int_0^{t_1}(t_1-s)^{-1+2\epsilon} \left(\int_{t_1}^{t_2}(t-s)^{-\epsilon} dt \right)^{2}ds\|A^{(\beta-2)/2}\|_{\mathcal{L}_2^0}^2\\
&\leq &ct_1^{2\epsilon} (t_2-t_1)^{2(1-\epsilon)}\|A^{(\beta-2)/2}\|_{\mathcal{L}_2^0}^2.
\end{eqnarray*}
For $3/2<\theta<2$,  \eqref{burk} and \eqref{dE} with $\ell=1$ and $m=\gamma$ yield  
\begin{eqnarray*}
\|I_1\|_{L^{2p}(\Omega;\dot H^q)}^2&=&\left\|\int_0^{t_1}A^{q/2}(E_{\gamma}(t_2-s)-E_{\gamma}(t_1-s) ) 
dW(s)\right\|_{L^{2p}(\Omega;H)}^2\\
&= &\left\|\int_0^{t_1}\left(\int_{t_1}^{t_2}A^{q/2}E_{\gamma}'(t-s)dt \right)dW(s)\right\|^{2}_{L^{2p}(\Omega;H)}\\
&\leq &\left\|\int_0^{t_1}\left(\int_{t_1}^{t_2}A^{(q+2-\beta)/2}E_{\gamma}'(t-s)dt \right)A^{(\beta-2)/2}dW(s)\right\|^{2}_{L^{2p}(\Omega;H)}\\
&\leq &c\int_0^{t_1} \left(\int_{t_1}^{t_2} (t-s)^{\theta-2}dt \right)^{2}ds\|A^{(\beta-2)/2}\|_{\mathcal{L}_2^0}^2\\
&\leq &c\int_0^{t_1}(t_1-s)^{2\theta-4} \left(\int_{t_1}^{t_2} dt \right)^{2}ds\|A^{(\beta-2)/2}\|_{\mathcal{L}_2^0}^2\\
&\leq &ct_1^{2\theta-3} (t_2-t_1)^2\|A^{(\beta-2)/2}\|_{\mathcal{L}_2^0}^2.
\end{eqnarray*}
To estimate $I_2$, we again use \eqref{burk} and \eqref{dE} as follow 
\begin{eqnarray*}
\|I_2\|_{L^{2p}(\Omega;\dot H^q)}^2
&\leq &\int_{t_1}^{t_2}\|A^{(q+2-\beta)/2}E_{\gamma}(t_2-s) A^{(\beta-2)/2}\|_{\mathcal{L}_2^0}^2ds\\
&\leq &c\int_{t_1}^{t_2}(t_2-s)^{2\theta-2}ds \|A^{(\beta-2)/2}\|_{\mathcal{L}_2^0}^2\\
&= &c(t_2-t_1)^{2\theta-1} \|A^{(\beta-2)/2}\|_{\mathcal{L}_2^0}^2\leq c(t_2-t_1)^2 T^{2\theta-3} \|A^{(\beta-2)/2}\|_{\mathcal{L}_2^0}^2,
\end{eqnarray*}
which completes the proof of \eqref{i2}.
%
\end{proof} 
We note that, by assumption \eqref{eta},    $\theta>1/2$ for $\beta\in[1,3] $ and $q= 0,\;1$. Moreover, if $q=1$ then $\theta-1/2=\eta$.
In the rest of the paper,  we will assume $\theta\neq3/2$ for the sake of simplicity.


\section{Regularity of the Solution} \label{sec:regularity}
In this section, we recall some properties of the nonlinear term $\phi(u)$ and discuss the regularity  in time of the exact solution $u$. The results are essential in the  error analysis discussed in the forthcoming sections.  For the first lemma, we refer to \cite{EL-1992,Larsson-2006, 12}.

\begin{lemma}\label{lem:6.2m} There is $C>0$ such that for all   $u$ and $v\in \dot H^1$, 
\begin{eqnarray}
\|\phi(u)\|&\leq& C \big(\|u\|+\|u\|_{\dot H^1}^3\big)\label{C_1m},\\
\|\phi(u)-\phi(v)\|&\leq& C\big(1+\|u\|_{\dot H^1}^2+\|v\|_{\dot H^1}^2\big)\|u-v\|_{\dot H^1},\label{C_2m}\\
\|A^{-1/2}(\phi(u)-\phi(v))\|&\leq& C\big(1+\|u\|_{\dot H^1}^2+\|v\|_{\dot H^1}^2\textbf{•})\|u-v\|,\label{C_3m}\\						
{(\phi(u)-\phi(v),u-v)}&\leq &{\|u-v\|^2.}\label{f-tr} 
\end{eqnarray}
Furthermore,  for $u$  and $v$ $\in\dot H^2$, there holds
\begin{equation}\label{g-tr}
\|\nabla P(\phi(u)-\phi(v))\|\leq  C\|u-v\|_{\dot H^1}\big(1+\|u\|_{\dot H^2}^2+\|v\|_{\dot H^2}^2\big).
\end{equation}
\end{lemma}
As a a result, we obtain 
\begin{lemma}\label{lem:6.2} Let $u$ and $v\in L^{8p}(\Omega;{\dot H^2})$, $p\geq 1$, with 
$\|u\|_{L^{8p}(\Omega;\dot H^1)}\leq R$. 
Then
\begin{eqnarray}
\|\phi(u)\|_{L^{2p}(\Omega;\dot H)}&\leq& C(R), \label{C_1}\\
{\|\nabla P(\phi(u)-\phi(v))\|_{L^{2p}(\Omega;\dot H)}} &\leq &{ C\|u-v\|_{L^{4p}(\Omega;\dot H^1)}\left(1+\|u\|_{L^{8p}(\Omega;\dot H^2)}^2+\|v\|_{L^{8p}(\Omega;\dot H^2)}^2\right)}.\label{C_3}
\end{eqnarray}
\end{lemma}
In our analysis, we shall assume that \eqref{main} admits a unique solution $u\in C([0,T];L^{16p}(\Omega;\dot H^\nu))$ when the initial data $u_0\in L^{16p}(\Omega;\dot H^\nu)$ with $\nu\in [2,4]$. It is worth noting that this assumption appears to be reasonable. Indeed, a similar assumption was substantiated for the classical stochastic Cahn--Hilliard equation  ($\alpha=1,\; \gamma=0$)  in  \cite[Theorem 2.2]{12}.  
%
\begin{theorem}\label{local}
Let $u_0\in L^{16p}(\Omega;\dot H^\nu)$, $p\geq 1$,  with $2\leq \nu\leq 4$. Then, for $\beta\in[1,3]$,  
\begin{equation}\label{holder}
\|u(t)-u(s)\|_{L^{{4p}}(\Omega;\dot H^q)}\leq c (t-s)^{\min\{\alpha(\nu -q)/4,\theta-1/2\}},\quad  
q=0,1,
\end{equation}
for all $0<s<t<T$, where $\theta$ is given by \eqref{theta}.
\end{theorem}
\begin{proof}We provide the proof for \eqref{holder} when $q=1$. 
From the solution representation \eqref{form-1}, we obtain for $h>0$:
\begin{eqnarray*}
u(t+h)-u(t)&=& \left\{ (E_{1-\alpha}(t+h)-E_{1-\alpha}(t))u_0\right\}\\
& &+\left\{\int_{t}^{t+h} E_0(s)AP\phi(u(t+h-s))ds +\int_0^{t}(E_0(t+h-s)-E_0(t-s))AP\phi(u(s))ds\right\}\\
& &+W_A(t+h)-W_A(t)\\
& &=:I_4+I_5+I_6.
\end{eqnarray*}
Now, the bound for $I_4$ is deduced from \eqref{dE} with $\ell=1$ as follows
\begin{eqnarray*}
\|I_4\|_{L^{4p}(\Omega;\dot H^1)}& \leq & \left\|\int_t^{t+h} A^{1/2} E_{1-\alpha}'(s)u_0\; ds\right\|_{L^{4p}(\Omega;\dot H)} \\
&\leq &c\int_t^{t+h} s^{\alpha(\nu-1)/4-1}\; ds\;\| u_0\|_{L^{4p}(\Omega;\dot H^\nu)}\\
  &\leq & c h^{\alpha(\nu-1)/4}.
\end{eqnarray*}
For $I_5$,   using   \eqref{dE} and \eqref{C_1} give
\begin{eqnarray*}
\|I_5\|_{L^{4p}(\Omega;\dot H^1)}& \leq & \left\|\int_{t}^{t+h}A^{3/2}E_0(t+h-s)P\phi(u(s))ds\right\|_{L^{4p}(\Omega;H)}\\
&&+\left\|\int_0^{t}\int_{t-s}^{t+h-s}A^{3/2}E_0'(\eta)\;d\eta P\phi(u(s))ds\right\|_{L^{4p}(\Omega;H)}\\
& \leq & \int_{t}^{t+h}s^{\alpha/4-1} \|P\phi(u(t+h-s))\|_{L^{4p}(\Omega;H)}ds\\
&&+\int_0^{t}((t-s)^{\alpha/4-1}-(t+h-s)^{\alpha/4-1})\|P\phi(u(s))\|_{L^{4p}(\Omega;H)}ds\\
& \leq & c h^{\alpha/4}.
\end{eqnarray*}
Estimate of $I_6$ is given in Lemma \ref{LL}.  The bound for $q=0$ can be derived analogously.
\end{proof}
\begin{remark}
The estimate in Theorem \ref{local} leads to 
\begin{eqnarray}\label{lip}
\|\phi(u(t))-\phi(u(s))\|_{L^{{2p}}(\Omega;\dot H^{q-1})}
&\leq &\left\| \; \|u(t)-u(s)\|_{\dot H^{q}}\big(1+\|u(t)\|^2_{\dot H^1}+\|u(s)\|^2_{\dot H^1}\big)\right\|_{L^{2p}(\Omega, \mathbb{R})}\nonumber\\
  & \leq & \|u(t)-u(s)\|_{L^{{4p}}(\Omega, \dot H^{q})}
  \left(1+\sup_{s\in[0,T]}\|u(s)\|^2_{L^{{8p}}(\Omega, \dot H^1)}\right)\\
  &\leq & c (t-s)^{\min\{{{\frac{\alpha(\nu-q)}{4},\theta-\frac{1}{2}\}}}},\nonumber 
 \end{eqnarray}
where  $q=0,1$ and $\beta\in[1,3]$.
\end{remark}


\section {Galerkin Semidiscrete Problem}\label{sec:FE}
In this section, we present the semidiscrete finite element scheme and derive related error estimates.  Let $\mathcal{G}_h$ ($0<h<1$) be a shape  regular  triangulation of the domain $\bar{\mathcal{D}}$ into triangles  $K$, 
with $h=\max_{K\in \mathcal{G}_h}h_{K},$ where $h_{K}$ denotes the diameter  of $K.$  
 Denote  by $S_h\subset H^1$  the FE space 
of continuous piecewise linear functions over the triangulation $\mathcal{G}_h$, 
and  $\dot S_h =\{\chi\in S_h:(\chi,1)=0\}$.
Then, the  semidiscrete problem reads: find $u_h(t)$, $w_h(t)\in \dot S_h$ such that 
\begin{equation}\label{3.1semi}
\left\{\begin{array}{ll}
(\cop u_h,\chi)+(\nabla w_h,\nabla\chi)=\partial_t^{-\gamma}P_h\dot W(t) & \forall \chi \in \dot S_h,\; t>0,\\
(w_h,\chi)=(\nabla u_h,\nabla\chi)+(\phi(u_h),\chi) & \forall \chi \in \dot S_h,\; t>0,\\
u_h(0)=u_{0h}, & \\
\end{array}\right.
\end{equation}
where $P_h:H\rightarrow \dot S_h$ is the orthogonal projection and $u_{0h}=P_h u_0$. Upon introducing the discrete operator $A_h:\dot S_h\to \dot S_h$ defined by 
\begin{equation*} \label{dL}
(A_h\varphi,\chi)=(\nabla \varphi,\nabla \chi)  \quad \forall \varphi,\chi\in \dot S_h,
\end{equation*}
the system of equations \eqref{3.1semi} is equivalent to
\begin{equation}\label{3.3}
\cop u_h+A_h^2u_h+A_hP_h\phi(u_h)=\partial_t^{-\gamma} P_h\dot W(t),\quad t>0,\quad u_h(0)=u_{0h}\in \dot S_h.
\end{equation}
 We observe that $A_h$ is both self-adjoint and positive definite, and for all $\chi \in \dot S_h$, we have $\|A_h^{1/2}\chi\| = \|\nabla \chi\| = \|\chi\|_{\dot H^1}$. Additionally, the relationship $P_h=P_hP$ holds true, given that $\dot S_h$ is a subspace of $\dot H$. We shall assume that $P_h$ remains bounded with respect to the $\dot H^1$-norm:
\begin{equation*}\label{Ps}
\|A_h^{1/2} P_h v\|\leq c\|A^{1/2}  v\|, \quad v\in \dot H^1.
\end{equation*}
This assumption remains applicable, particularly in cases where the mesh $\mathcal{G}_h$ is quasi-uniform. 

Let ${A_h^{-1}}: \dot S_h \to \dot S_h$ represent the inverse of $A_h$. We extend $A_h^{-1}$ to the entire  space $H$ by $A_h^{-1}f = A_h^{-1}P_h f$ for $f \in H$ so that $v_h=A_h^{-1}f$ if and only if $A_hv_h=P_hf$. In other words, 
$$
v_h\in \dot S_h,\quad (\nabla v_h, \nabla \chi)=(f, \chi),\quad \forall \chi \in \dot S_h.
$$ 
Thus, $A_h^{-1}$ is self-adjoint, positive semidefinite on $H$, and positive definite on $\dot S_h$. Additionally, it is worth recalling the established fact (see \cite{Larsson-2006}) that
\begin{equation}\label{Ah}
\|A_h^{-1/2}P_hv\|\leq c\|v\|_{\dot H^{-1}},\quad v\in H.
\end{equation}
Now,  by the Duhamel's principle, the solution of \eqref{3.3} is represented by  
\begin{equation}\label{form-1d}
u_h(t)=E_{1-\alpha,h}(t)P_hu_0-\int_0^t {  E}_{0,h}(t-s)A_hP_h\phi(u_h(s))\,ds+\int_0^tE_{\gamma,h}(t-s)P_hdW(s),
\end{equation}
where $E_{m,h}(t):\dot S_h\to\dot S_h$ is given by
$$
 E_{m,h}(t) = \frac{1}{2\pi i}\int_{\Gamma_{\theta,\delta}}e^{zt} z^{-m}(z^\alpha I+A_h^2)^{-1} \,dz.
$$
Since  $A_h$ is selfadjoint and positive definite,  the estimates
in Lemmas \ref{eE}   remains valid for $A_h$,  uniformly in $h$.

For the error analysis, we introduce the operators $K_h(z)=(z^\alpha I+A_h^2)^{-1}P_h-(z^\alpha I +A^2)^{-1}$ and $\bar {K}_h(z)= (z^\alpha I+A_h^2)^{-1}A_hP_h-(z^\alpha I +A^2)^{-1}A$. Some estimates related to these operators are presented in the following lemmas. Their proofs are given in  \cite[Lemma 4.1, Lemma 4.2]{MK-2021}.
\begin{lemma} \label{S_h} The following estimate holds for all $z\in\Sigma_\theta$,
\begin{equation*}
\|K_h(z)v\|+h\|\nabla K_h(z)v\| \leq ch^2 |z|^{-\alpha(1+\nu/2)/2}\left\|v\right\|_{\dot{H}^\nu}, \quad\nu\in[0,2].
\end{equation*}
\end{lemma}

\begin{lemma}\label{barS_h}  The following estimate holds for all $z\in\Sigma_\theta$  
\begin{equation*}
\|\bar{K}_h(z)v\|+h\|\nabla \bar{K}_h(z)v\| \leq ch^2 |z|^{-\alpha \textcolor{red}{\nu}/4}\left\|v\right\|_{\dot{H}^\nu},{\quad\nu\in[0,2]}.
\end{equation*}
\end{lemma}

Further, we define the operators $F_{1-\alpha,h}(t):=E_{1-\alpha,h}(t)P_h-E_{1-\alpha}(t)$ and $\bar{F}_{0,h}(t):={E}_{0,h}(t)A_hP_h-{E}_{0}(t)A$. Then, by Lemmas \ref{S_h} and \ref{barS_h}, we have  
\begin{equation}\label{bar-F}
\left\|\bar{F}_{0,h}(t)v\right\| +h\left\|\nabla\bar{F}_{0,h}(t)v\right\|\leq c\int_{\Gamma_{\theta,\delta}}e^{Re(z)t} \left( \left\|\bar{K}_hv\right\|+h\left\|\nabla \bar{K}_hv\right\|\right)  \,|dz|\leq ch^2 {t^{\alpha\nu/4-1}\left\|v\right\|_{\dot{H}^\nu},\quad v\in \dot{H}^\nu},
\end{equation}
and 
\begin{equation}\label{F_h}
\left\|F_{1-\alpha,h}(t)v\right\|+h \left\|\nabla F_{1-\alpha,h}(t)v\right\|
\leq ch^2 t^{-\alpha(2-\nu)/4}\left\|v\right\|_{\dot{H}^\nu},\quad v\in \dot{H}^\nu,
\end{equation}
for $\nu\in[0,2]$. Finally, we  let $F_{\gamma,h}(t):=E_{\gamma,h}(t)P_h-E_{\gamma}(t)$. Then we have the following result.
\begin{lemma} Let $s\in [0,1]$ and $r,\nu\in [0,2]$. Then 
\begin{equation}\label{Fgamma}
\|A^{s/2}F_{\gamma,h}v \|\leq c h^{2-s-r} t^{\alpha (r+2+\nu)/4+\gamma-1}\|A^{\nu/2}v\|, 
\end{equation}
where  $r+s\leq 2$ and $r+\nu\leq 2$.
%
\end{lemma}
Now, we introduce  the intermediate solution $\tilde u_h$ defined by 
\begin{equation}\label{3.1inter}
\left\{\begin{array}{ll}
(\cop \tilde u_h,\chi)+(\nabla \tilde w_h,\nabla\chi)=\partial_t^{-\gamma} P_h\dot W(t) & \forall \chi \in \dot S_h,\; t>0,\\
(\tilde w_h,\chi)=(\nabla \tilde u_h,\nabla\chi)+(\phi_h(u),\chi) & \forall \chi \in \dot S_h,\; t>0,\\
\tilde u_h(0)=P_h u_{0}. & \\
\end{array}\right.
\end{equation}
Equivalently, $\tilde u_h$ satisfies
\begin{equation}\label{3.33}
\cop \tilde u_h+A_h^2\tilde u_h+A_hP_h\phi(u)=\partial_t^{-\gamma} P_h\dot W(t),\quad t>0,\quad \tilde u_h(0)=P_h u_{0}\in \dot S_h.
\end{equation}
In our analysis, we shall assume that 
\begin{equation}\label{u_hBound}
\sup_{t\in[0,T]} \mathbb{E} \|A_h u_h(t)\|_{\dot H}^{16p}<C \quad \mbox{and}\quad  \sup_{t\in[0,T]} \mathbb{E} \|A_h \tilde u_h(t)\|_{\dot H}^{16p}<C.
\end{equation} 

To derive a bound for the error $e(t):=u_h(t)-u(t)$, we consider the splitting $e(t)= e_1(t)+e_2(t)$, where $ e_1(t):=u_h-\tilde u_h$ and $e_2(t):=\tilde u_h-u $. We start by showing  a bound for  $e_2(t)$.
\begin{lemma}\label{thm:6.5}
Let $u_0\in L^{16p}(\Omega;\dot H^2)$, $p\geq 1$, and choose $u_{0h}=P_hu_0$. Let 
 $u$ and $\tilde u_h$ be the solutions of \eqref{main} and \eqref{3.33} over $[0,T]$, respectively. If  
 $\|A^{(\beta-2)/2}\|_{{\cal L}_2^0}<\infty$ and $\eta>0$, then
 \begin{equation*}\label{6.11-n}
\|e_2(t)\|_{L^{{2p}}(\Omega,H)}\leq c h^{\min\{2,\beta\}-r},\quad t\in (0,T],
\end{equation*}
where 
 \begin{eqnarray} \label{rr}
r=\begin{cases}
\dfrac{4}{\alpha}(\frac{1-\alpha}{2}-\gamma )+\epsilon & \text{if }\gamma+\alpha/2< \frac{1}{2}\\ 
  \epsilon & \text{if }\gamma+\alpha/2= \frac{1}{2},\\
  0 & \text{if }\gamma+\alpha/2> \frac{1}{2},
\end{cases}
\end{eqnarray}
if $\beta\in[1,2]$, and 
\begin{eqnarray} \label{rr1}
r=\begin{cases}
\dfrac{4}{\alpha}(\frac{2-\alpha\beta}{4}-\gamma )+\epsilon & \text{if }\gamma+\alpha\beta/4< \frac{1}{2}\\ 
  \epsilon & \text{if }\gamma+\alpha\beta/4= \frac{1}{2},\\
  0 & \text{if }\gamma+\alpha\beta/4> \frac{1}{2},
\end{cases}
\end{eqnarray}
if  $\beta\in(2,3]$. Here $c$ may depend on $T$.
\end{lemma}
\begin{proof}
In view of  \eqref{main},  \eqref{3.33} and the property $P_hP=P_h$, we may write $e_2(t)$ as 
\begin{equation*}\label{intt}
e_2(t)= F_{1-\alpha,h}(t)u_0
-\int_0^t  \bar{F}_{0,h}(t-s) P\phi(u(s))\,ds
+\int_0^t {F}_{\gamma,h}(t-s) dW(s).
\end{equation*} 
Using \eqref{bar-F}, \eqref{F_h} and {\eqref{C_3}}, we then deduce
\begin{eqnarray*}
\|e_2(t)\|_{L^{{2p}}(\Omega,H)} &\leq &ch^2\|u_0\|_{L^{2p}(\Omega,\dot H^2)}
+\int_{0}^{t}\|\bar{F}_{0,h}(t-s) P\phi(u(s))\|_{L^{2p}(\Omega,H)}\,ds\\
& &+\int_{0}^{t}\|{F}_{\gamma ,h}(t-s) dW(s)\|_{L^{2p}(\Omega,H)}\,ds\\
&\leq &ch^2\|u_0\|_{L^{2p}(\Omega,\dot H^2)}
+ch^2\int_{0}^{t}(t-s)^{\alpha/4-1} \|\nabla P\phi(u(s))\|_{L^{2p}(\Omega,H)}\,ds\\
& &+\int_{0}^{t}\|{F}_{\gamma ,h}(t-s) dW(s)\|_{L^{2p}(\Omega,H)}\,ds
\end{eqnarray*}
For $\beta\in[1,2]$, we have
\begin{eqnarray*}
\int_{0}^{t}\|{F}_{\gamma ,h}(t-s) dW(s)\|_{L^{2p}(\Omega,H)}\,ds
&\leq & h^{\beta-r}\left(\int_{0}^{t}(t-s)^{2[\alpha(2+r)/4+\gamma-1]}\,ds\right)^{1/2}\|A^{(\beta-2)/2}\|_{\mathcal{L}_2^0}\\
&\leq &  h^{\beta-r} t^{\alpha(2+r)/4+\gamma-1/2}\|A^{(\beta-2)/2}\|_{\mathcal{L}_2^0}.
\end{eqnarray*}
 The choice  of $r$ in \eqref{rr} ensures that  $\alpha(2+r)/4+\gamma-1/2>0$. Thus,  we have 
$\|e_2(t)\|_{L^{2p}(\Omega,H)}  \leq  ch^{\beta-r}.$ 
For $\beta\in(2,3]$, we use \eqref{Fgamma} with $s=0$ and $\nu=\beta-2$ to get
\begin{eqnarray*}
\int_{0}^{t}\|{F}_{\gamma ,h}(t-s) dW(s)\|_{L^{2p}(\Omega,H)}\,ds
&\leq & h^{2-r}\left(\int_{0}^{t}(t-s)^{2[\alpha(r+\beta)/4+\gamma-1]}\,ds\right)^{1/2}\|A^{(\beta-2)/2}\|_{\mathcal{L}_2^0}\\
&\leq &  h^{2-r} t^{\alpha(r+\beta)/4+\gamma-1/2}\|A^{(\beta-2)/2}\|_{\mathcal{L}_2^0},
\end{eqnarray*}
where $r$ is given by \eqref{rr1}.
\end{proof}
In the next lemma we derive an estimate  involving  in the $H^1$-norm of $e_1(t)$.
\begin{lemma}\label{error_integral}
Let $u_h$ and $\tilde u_h$ be a solution of \eqref{3.3} and \eqref{3.33}, respectively, with $u_0\in L^{16p}(\Omega;\dot H^2)$. Then 
\begin{equation}\label{interm-err-integ}
\int_0^t(t-s)^{\alpha-1}\|\nabla e_1(s)\|^2_{L^{2p}(\Omega, \dot H)}\;ds\leq  Ch^{2\min\{2,\beta\}-2r}   ,
\end{equation}
where $r$ is given by \eqref{rr} if  $\beta \in[1,2]$  and \eqref{rr1} if  $\beta \in(2,3]$.
\end{lemma}
\begin{proof}
From \eqref{3.33} and \eqref{3.3}, $e_1$ satisfies
\begin{eqnarray}\label{errorTilde}
 \cop e_1(t)+A_h^2 e_1(t)+A_hP_h(\phi(u_h)-\phi(u))=0,\quad\ e_1(0)=0.
 \end{eqnarray}
Taking the inner product with $A^{-1}_h e_1(t)$ and using \eqref{f-tr}, we get
\begin{eqnarray*}
 (\cop A^{-1/2}_he_1(t),A^{-1/2}_he_1(t)) +\|\nabla e_1(t)\|&=&(\phi(u)-\phi(\tilde u_h),e_1(t))+(\phi(\tilde u_h)-\phi(u_h),e_1(t))\\
 & \leq &\frac{1}{2}\|\nabla e_1(t)\|^2+\frac{1}{2}\|A^{-1/2}( \phi(u)-\phi(\tilde u_h))\|^2+\|e_1(t)\|^2.
 \end{eqnarray*}
Using the fact 
$(\cop A^{-1/2}_he_1(t), A^{-1/2}_he_1(t))\geq \frac{1}{2}\cop\|A^{-1/2}_he_1(t)\|^2 $ and $\|e_1(t)\|^2\leq \|\nabla e_1(t)\| \|e_1(t)\|_{\dot H^{-1}}$, we obtain
$$  \frac{1}{2}\cop\|A^{-1/2}_he_1(t)\|^2 +\|\nabla e_1(t)\|
  \leq \frac{1}{2}\|\nabla e_1(t)\|^2+\frac{1}{2}\|A^{-1/2}( \phi(u)-\phi(\tilde u_h))\|^2+\frac{1}{3}\|\nabla e_1(t)\|^2+\frac{3}{4}\|e_1(t)\|_{\dot H^{-1}}^2,$$
 which can be written as 
 $$ \cop\|A^{-1/2}e_1(t)\|^2 +\|\nabla e_1(t)\|
  \leq c\|A^{-1/2}( \phi(u)-\phi(\tilde u_h))\|^2+c\|e_1(t)\|_{\dot H^{-1}}^2 .$$
 Applying a fractional Gronwall's inequality (see \cite[Lemma 2]{Alikhanov-2010}), we deduce that
 $$\|A^{-1/2}e_1(t)\|^2+\int_0^t(t-s)^{\alpha-1}\|\nabla e_1(s)\|^2\;ds\leq C\int_0^t (t-s)^{\alpha-1}\|A^{-1/2}(\phi(\tilde u_h(s))- \phi(u(s)))\|^2 ds.$$
By \eqref{C_3m}, we have
 $$\|A^{-1/2}e_1(t)\|^2+\int_0^t(t-s)^{\alpha-1}\|\nabla e_1(s)\|^2ds\leq C\int_0^t (t-s)^{\alpha-1}\|e_2(s)\|^2 (1+\|\tilde u_h(s)\|^4_{\dot H^1}+\|u(s)\|_{\dot H^1}^4) ds,$$
and  therefore, 
\begin{eqnarray*}
\|A^{-1/2}e_1(t)\|^2_{L^{2p}(\Omega, H)}&+&\int_0^t(t-s)^{\alpha-1}\|\nabla e_1(s)\|^2_{L^{2p}(\Omega, \dot H)}\;ds\\
&\leq& C\int_0^t (t-s)^{\alpha-1}\|e_2(s)\|^2_{L^{4p}(\Omega, \dot H)} \big(1+\|\tilde u_h(s)\|^4_{L^{8p}(\Omega,\dot H^1) }+\|u(s)\|_{L^{8p}(\Omega,\dot H^1 )}^4\big) ds,
\end{eqnarray*}
showing that
\begin{eqnarray*}
\small
\int_0^t(t-s)^{\alpha-1}\|\nabla e_1(s)\|^2_{L^{2p}(\Omega, H)}\;ds&\leq & C\int_0^t(t-s)^{\alpha-1} \|e_2(s)\|^2_{L^{4p}(\Omega, H)}ds\\
 & & +C\int_0^t(t-s)^{\alpha-1} \|e_2(s)\|^2_{L^{4p}(\Omega, H)} \|\tilde u_h(s)\|^4_{L^{8p}(\Omega,\dot H^1) }ds\\
 & &+ C\int_0^t(t-s)^{\alpha-1} \|e_2(s)\|^2_{L^{4p}(\Omega, H)}\|u(s)\|_{L^{8p}(\Omega,\dot H^1 )}^4 ds.
\end{eqnarray*}
Using the estimate in Lemma \ref{thm:6.5}, we obtain
\begin{eqnarray*} 
\int_0^t(t-s)^{\alpha-1}\|\nabla e_1(s)\|^2_{L^{2p}(\Omega, \dot H)}\;ds &\leq & C\int_0^t (t-s)^{\alpha-1}h^{\min\{2\beta,4\}-2r}ds\\
 &\leq & Ch^{\min\{2\beta,4\}-2r},
\end{eqnarray*}
which completes the proof.
\end{proof}
%
\begin{lemma}\label{error_tilde_u}
Let $u_h$ and $\tilde u_h$ be a solution of \eqref{3.3} and \eqref{3.33}, respectively, with $u_0\in L^{16p}(\Omega;\dot H^2)$. Then 
\begin{equation}\label{interm-err}
\|e_1(t)\|_{L^{2p}(\Omega, H)}\leq Ch^{\min\{2,\beta\}-r},   
\end{equation}
where $r$ is given by \eqref{rr} if  $\beta \in[1,2]$  and by \eqref{rr1} if  $\beta \in(2,3]$.
\end{lemma}
\begin{proof}
In view of  \eqref{errorTilde}, 
$e_1(t)=\int_0^t { E}_{0,h}(t-s)A_hP_h(\phi(u)-\phi(u_h))\,ds,$ and so 
\begin{eqnarray*}
\|e_1(t)\|_{L^{2p}(\Omega, H)}& \leq & \int_0^t \|{ E}_{0,h}(t-s)A_hP_h(\phi(u)-\phi(\tilde u_h))\|_{L^{2p}(\Omega, \dot H)}\,ds\\
   & &+ \int_0^t \|{ E}_{0,h}(t-s)A_hP_h(\phi(\tilde u_h)-\phi(u_h))\|_{L^{2p}(\Omega, \dot H)}\,ds\\
   & =:& I_1+I_2 .
\end{eqnarray*}
The first term $I_1$ is bounded  as follow:
\begin{eqnarray*}
I_1 &\leq & \int_0^t (t-s)^{\alpha/4-1}\|A^{-1/2}(\phi(u)-\phi(\tilde u_h))\|_{L^{2p}(\Omega, \dot H)} \; ds\\
 &\leq &C \int_0^t (t-s)^{\alpha/4-1}\|e_2(s)\|_{L^{4p}(\Omega, \dot H)} (1+\|\tilde u_h(s)\|^2_{L^{8p}(\Omega,\dot H^1)}+\|u(s)\|^2_{L^{8p}(\Omega,\dot H^1 ) })\; ds \\
 &\leq &  
 C\int_0^t (t-s)^{\alpha/4-1}h^{\min\{2,\beta\}-r}ds \\
 &\leq & Ch^{\min\{2,\beta\}-r}. 
\end{eqnarray*}
For $I_2$, we use {\eqref{C_3}}  to get
\begin{eqnarray*}
I_2 & \leq & \int_0^t (t-s)^{3\alpha/4-1}\|\nabla P(\phi(\tilde u_h)-\phi(u_h))\|_{L^{2p}(\Omega, \dot H)}\,ds\\
& \leq & C\int_0^t (t-s)^{3\alpha/4-1}\|\nabla e_1(s)\|_{L^{4p} (\Omega, \dot H)}(1+{\|A_h \tilde u_h(s)\|^2_{L^{8p}(\Omega,\dot H)}+\| A_h u_h(s)\|^2_{L^{8p}(\Omega,\dot H )} })\; ds\\
& \leq & C\left(\int_0^t (t-s)^{\alpha-1}\|\nabla e_1(s)\|^2_{L^{4p} (\Omega, \dot H)}\; ds\right)^{1/2}\\
& &\left(\int_0^t (t-s)^{\alpha/2-1}(1+\|A_h\tilde u_h(s)\|^4_{L^{8p}(\Omega,\dot H)}+\|A_h u_h(s)\|^4_{L^{8p}(\Omega,\dot H ) })\; ds\right)^{1/2}\\
& \leq & C\left(\int_0^t (t-s)^{\alpha/2-1}\|\nabla e_1(s)\|^2_{L^{4p} (\Omega, \dot H)}\; ds\right)^{1/2}\\
& \leq &  Ch^{\min\{2,\beta\}-r}.
\end{eqnarray*}
The last inequality is obtained by using \eqref{interm-err-integ}.
\end{proof}
Based on the previous  lemmas and the triangle inequality, we establish our main theorem.
\begin{theorem}\label{thm:6.51}
Let $u_0\in L^{16p}(\Omega;\dot H^2)$ and $u_{0h}=P_hu_0$. Let 
 $u$ and $u_h$ be the solutions of \eqref{main} and \eqref{3.3} over $[0,T]$, respectively. Then,  if $\|A^{(\beta-2)/2}\|_{{\cal L}_2^0}<\infty$ and $\eta>0$, there holds  
\begin{equation*}\label{6.11-n}
\|u_h(t)-u(t)\|_{L^{2p}(\Omega,\dot H)}\leq c h^{\min\{2,\beta\}-r},\quad t\in (0,T],
\end{equation*}
where $c$ may depend on $T$, and $r$ is given by \eqref{rr} if  $\beta \in[1,2]$  and by \eqref{rr1} if  $\beta \in(2,3]$.
\end{theorem}
\begin{remark}
For trace class noise ($\beta=2$), the convergence rate is $O(h^{4-\frac{4}{\alpha}(\frac{1}{2}-\gamma)-\epsilon})$ for $\gamma+\frac{\alpha}{2}<1/2$ and $O(h^2)$ for $\gamma+\frac{\alpha}{2}>1/2$. The former is due to the limited smoothing property of solution operator $E_\gamma(t)$.
In the limiting case $\alpha \to 1^-$ and $\gamma\to 0^+$, the convergence rate is $O(h^{2-\epsilon})$ which agrees with that for the classical stochastic Cahn--Hilliard equation \cite{10}.
For smoother noise ($\beta>2$), the convergence rate is $O(h^{2+\beta-\frac{4}{\alpha}(\frac{1}{2}-\gamma)-\epsilon})$ for $\gamma+\frac{\alpha\beta}{2}<1/2$ and $O(h^2)$ for $\gamma+\frac{\alpha\beta}{2}>1/2$. In the limiting case $\alpha \to 1^-$ and $\gamma\to 0^+$, we clearly have an $O(h^{2})$ convergence rate, which confirms the result obtained for the classical stochastic Cahn--Hilliard equation \cite{12}.

\end{remark}



\section{Fully Discrete Scheme} \label{sec:time}
For the time discretization of \eqref{3.3},   we propose a fully discrete scheme  using  a convolution quadrature (CQ) generated by the backward Euler (BE) method \cite{Lubich-2004}.  To describe  the  scheme, we divide the time interval $[0,T]$ into $N$ uniform partitions  with grid points $t_j=j\tau,\; j=0,1,\dots , N,$ and a time step size $\tau=T/N$.

Now, let $g^0=0$ and 
$ g^k=\tau^{-1}P_h\Delta W^k,\; \mbox{where }\Delta W^k=W(t_k)-W(t_{k-1}),\; k=1,\dots, N.$ 
 Also, set $\op:=\partial_t \partial_t^{\alpha-1}$ to be the Riemann--Liouville derivative in time of order $\alpha\in (0,1)$. 
 
%
%
Upon using the relation 
$\cop\varphi(t)=\op(\varphi(t)-\varphi(0))$, 
the fully discrete scheme is to find $U_h^n\in \dot S_h$, $n=1,2\cdots N$, such that
\begin{equation}\label{fully-1}
\dop (U_h^n -U_h^0) + A_h^2U_h^n = -A_hP_h\phi(U_h^{n}) +\partial_\tau^{-\gamma}g^n\quad \text{ with }\; U_h^0=P_hu_0.
\end{equation}
The Riemann--Liouville derivative/integral $ \partial_t^{\ell},\; \ell=\alpha,-\gamma$, are approximated by the following CQ   
\begin{eqnarray}\label{QC}
\partial_t^{\ell} \varphi(t_n)\approx\partial_\tau^{\ell} \varphi_n :=\tau^{-\ell}\sum_{j=0}^n a_{n-j}^{(\ell)}\varphi_j,\quad\mbox{where}\; \sum_{j=0}^\infty a_j^{(\ell)}\xi^j
=(\delta(\xi))^\ell,\quad a_j^{(\ell)}=(-1)^j\left(\begin{array}{c}\ell\\j\end{array}\right),
\end{eqnarray}
where  $\varphi_j= \varphi(t_j)$ and $\delta(\xi)=1-\xi$.
  Then, the numerical scheme \eqref{fully-1} can be  expanded as  
\begin{equation} \label{fully-2}
\tau^{-\alpha}\sum_{j=0}^n a_{n-j}^{(\alpha)}  (U^j_h -U^0_h )+A_h^2 u^n_h= -A_h\phi_h(U_h^{n})+ \tau^{\gamma}\sum_{j=0}^n a_{n-j}^{(-\gamma)} g^j,\quad n\geq 1,
\end{equation}
where $\phi_h=P_h\phi$. 
Following the analysis in \cite[Section 4]{BYZ-2019}, the discrete solution can be written as 
\begin{equation}\label{semi-1c}
U_h^n = \bar U_h^n - \tau \sum_{j=1}^n R_{n-(j-1)} A_h\phi_h(U_h^{j})+ \tau \sum_{j=1}^n Q_{n-(j-1)} g^j,\quad n\geq 1,
\end{equation}
where $\bar U_h^n$  is the discrete solution to the homogeneous problem in \eqref{fully-2} and the operators $R_j$ and $Q_j$ satisfy
\begin{equation} \label{R_j}
\sum_{j=0}^\infty R_j\zeta^j=\widetilde{R}(\zeta)\quad \mbox{ with } \quad \widetilde{R}(\zeta)=1+\zeta(\tau^{-\alpha}\delta(\zeta)^\alpha+A_h^2)^{-1}\tau^{-1},
\end{equation}
and 
\begin{equation} \label{R_j}
\sum_{j=0}^\infty Q_j\zeta^j=\widetilde{Q}(\zeta)\quad \mbox{ with } \quad \widetilde{Q}(\zeta)=1+\zeta(\tau^{-\alpha}\delta(\zeta)^\alpha+A_h^2)^{-1}\tau^{\gamma-1}\delta(\zeta)^{-\gamma}.
\end{equation} 
Equivalently, \eqref{semi-1c} can be expressed as 
\begin{equation}\label{fully-3}
U_h^n = \bar U_h^n -  \sum_{j=0}^{n-1}  R_{n-j} \int_{t_{j}}^{t_{j+1}}A_h\phi_h(U_h^{j+1})\;dt+  \sum_{j=0}^{n-1} Q_{n-j} \int_{t_{j}}^{t_{j+1}} \;P_hdW(t).
\end{equation}
Estimates involving of the operators $R_j$ and $Q_j$ are given in the following lemmas. 
The proofs follows the arguments presented in \cite[Lemma 4.6]{BYZ-2019}.

\begin{lemma} \label{PQ} The following estimates hold for $n=0,1,2,\ldots,$
\begin{eqnarray}
 \|A_h^{s/2} R_nP_h\| &\leq & c t_{n+1}^{\alpha (1- s/4)-1},\quad s\in [0,4]\label{R_n2}\\
  \|A_h^{s/2} ( E_{0,h}(t_n)-R_n)P_h\| & \leq & c t_{n+1}^{\alpha (1- s/4)-2}\tau, \quad s\in [0,1].
\end{eqnarray} 
\end{lemma}
\begin{lemma} \label{Q_n} The following estimate holds for $s\in [0,1]$,
\begin{equation}
\quad \| A_h^{s/2} ( E_{\gamma,h}(t_n)-Q_n)P_h \|\leq c t_{n+1}^{\alpha (1-s/4)+\gamma-2}\tau,  
\quad n=0,1,2, \ldots.
\end{equation} 
\end{lemma}

To conduct our analysis, we shall employ the subsequent lemma to assess  the error.
\begin{lemma}\label{t_nANDa_n}
 The following inequality holds
$$\tau^\alpha a_{n}^{(-\alpha)}\geq \tau \dfrac{t_n^{\alpha-1}}{\Gamma(\alpha)}.$$
\end{lemma}
\begin{proof}
Starting from the definition of $a_{n}^{(-\alpha)}$, we have 
\begin{eqnarray*}
\tau^\alpha a_{n}^{(-\alpha)}&=&\tau^\alpha (-1)^{n}\left(\begin{array}{c}-\alpha\\j\end{array}\right)=\tau^\alpha \dfrac{\Gamma(n+\alpha)}{n\Gamma(n)\Gamma(\alpha)}.\\
\end{eqnarray*}
Then, a use of \cite[inequality (3.3)]{LN2013} yields 
\begin{eqnarray*}
\tau^\alpha a_{n}^{(-\alpha)}& \geq & \dfrac{ \tau^\alpha}{n\Gamma(\alpha)}\left(\dfrac{n}{n+\alpha}\right)^{1-\alpha}n^\alpha\\
   & = &\dfrac{ t_n^\alpha }{(n+\alpha)\Gamma(\alpha)}\left(\dfrac{n+\alpha}{n}\right)^{\alpha}\\
   &\geq & \dfrac{ t_n^\alpha }{2n\Gamma(\alpha)}\geq \dfrac{ t_n^{\alpha-1} \tau }{\Gamma(\alpha)},\\
\end{eqnarray*} 
which completes the proof.
\end{proof}

For the error analysis, we introduce the intermediate solution $\tilde u_h^n$ defined by 
\begin{equation}\label{DiscTldu}
\dop (\tilde u_h^n -\tilde u_h^0) + A_h^2\tilde u_h^n = -A_hP_h\phi(u_h(t_{n}) +\partial_\tau^{-\gamma}g^n\quad \text{ with }\; \tilde u_h^0=P_hu_0.
\end{equation} 
Then, in an expanded form, 
\begin{equation}\label{DiscTldu-3}
\tilde u_h^n = \bar U_h^n -  \sum_{j=0}^{n-1}  R_{n-j} \int_{t_{j}}^{t_{j+1}}A_h\phi_h( u_h(t_{j+1}))\;dt+  \sum_{j=0}^{n-1} Q_{n-j} \int_{t_{j}}^{t_{j+1}} \;P_hdW(t).
\end{equation}
Similar to \eqref{u_hBound}, we shall assume that 
\begin{equation}\label{u^nBound}
\sup_{0<j\leq N} \mathbb{E} \|A_h U_h^j\|_{\dot H}^{16p}<C \quad \mbox{and}\quad  \sup_{0<j\leq N} \mathbb{E} \|A_h \tilde u_h^j\|_{\dot H}^{16p}<C.
\end{equation} 

Next, we let $e^n_1=u_h(t_n)-\tilde u_h^n $ and $e^n_2=\tilde u_h^n- U_h^n$, and split the fully discrete error into three components: $u(t_n)-U_h^n=u(t_n)-u_h(t_n)+ e^n_1+ e^n_2$.  The focus is now on establishing an estimate for $e^n_1$. The subsequent lemmas are essential in achieving this goal. 
\begin{lemma}\label{lem:a3}
Let $\beta\in[1,3]$ and $u_0\in L^{16p}(\Omega;\dot H^2)$. If $\eta>0$,  then
\begin{equation}
\sum_{j=0}^{n-1}\int_{t_j}^{t_{j+1}}\left\|E_{0,h}(t_n-t)A_h\phi_h(u_h(t))-R_{n-j}A_h\phi_h(u_h(t_{j+1})))\right\|_{L^{2p}(\Omega; \dot H)}dt
\leq c \tau^{\min\{\alpha/2,\; \eta+\alpha/4,\; {\alpha+\gamma-1/2}\}}.
\end{equation}
\end{lemma}
\begin{proof}
See Appendix A.
\end{proof}
%
\begin{lemma}\label{lem:a2}
For {$\beta\in [1,2]$}, there holds
\begin{equation}\label{mis}
\left\|\sum_{j=0}^{n-1}\int_{t_j}^{t_{j+1}}(E_{\gamma,h}(t_n-t)-Q_{n-j})P_hdW(t)\right\|_{L^{2p}(\Omega;\dot H)}
\leq c\tau^\mu t_n^{\max\{\sigma-1,0\}}\|A^{(\beta-2)/2}||_{\mathcal{L}_0^2},
\end{equation}
where $\sigma=\eta+\alpha/4$ and
\begin{equation}\label{mu}
\mu=\left\{\begin{array}{ll} \sigma, & \sigma <1,\\ 1-\epsilon, &\sigma=1,\\1,&\sigma>1.\end{array}\right.
\end{equation}
\end{lemma}
\begin{proof}
See Appendix B.
\end{proof}

\begin{remark}\label{B23}
For the case where {$\beta\in (2,3]$},  we can argue as in the proof of the previous lemma to derive  the following result: 
$$
\left\|\sum_{j=0}^{n-1}\int_{t_j}^{t_{j+1}}(E_{\gamma,h}(t_n-t)-Q_{n-j})P_hdW(t)\right\|_{L^{2p}(\Omega;\dot H)}
\leq c\tau^{\zeta}t_n^{\max\{\xi-1,0\}}\|A^{(\beta-2)/2}||_{\mathcal{L}_0^2},
$$
where $\xi:=\alpha+\gamma-1/2$ and 
\begin{equation}\label{zeta}
\zeta=\left\{\begin{array}{ll} \xi, & \xi <1,\\ 1-\epsilon, &\xi=1,\\1,&\xi>1.\end{array}\right.
\end{equation}
\end{remark}
Now we are ready to prove an  estimate for $e^n_1$ in $L^{2p}(\Omega;\dot H)$-norm.
\begin{lemma}\label{e_1}
Let $u_h(t_n)$ and $\tilde u_h^n$ be the solutions of \eqref{form-1d} and \eqref{DiscTldu-3}, respectively, with 
$u_0\in L^{16p}(\Omega;\dot H^2)$. 
Assume $\|A^{(\beta-2)/2}\|_{{\cal L}_0^2}<\infty$ for some $\beta\in [1,3]$ and $\eta>0$. Then
\begin{equation}\label{fully-1-e1}
\|e^n_1\|_{L^{2p}(\Omega; \dot H)}\leq c(h^2+\tau^{\min\{\alpha/2,\;\mu,\;\zeta\}}),
\end{equation}
where $\mu$ and $\zeta$ are given by \eqref{mu} and \eqref{zeta}, respectively.
\end{lemma}
\begin{proof}  
 Noting the representation of the semidiscrete solution 
\begin{equation*}\label{semi-1d}
u_h(t_n) = E_{1-\alpha,h}(t_n) u_h^0 -  \sum_{j=1}^n  \int_{t_{j-1}}^{t_j} E_{0,h}(t_n-t)A_h\phi_h(u_h(t))\;dt+  \sum_{j=1}^n  \int_{t_{j-1}}^{t_j}E_{\gamma,h}(t_n-t) \;P_hdW(t),
\end{equation*}
we have, by the triangle inequality,
\begin{eqnarray*}
\|e^n_1\|_{L^{2p}(\Omega;H)}&\leq &\|E_{1-\alpha,h}(t_n)u_h^0-\bar U_h^n\|_{L^{2p}(\Omega;H)}\\
& & +\left\|\sum_{j=0}^{n-1}\int_{t_j}^{t_{j+1}}(E_{0,h}(t_n-t)A_h\phi_h(u_h(t))-R_{n-j}A_h\phi_h(u_h(t_{j+1})))\;dt\right\|_{L^{2p}(\Omega;\dot H)}\\
& &+\left\|\sum_{j=0}^{n-1}\int_{t_j}^{t_{j+1}}(E_{\gamma,h}(t_n-t)-Q_{n-j})P_hdW(t)\right\|_{L^{2p}(\Omega;\dot H)}.
\end{eqnarray*}
Then, using Lemmas \ref{lem:a3} and  \ref{lem:a2}, and the following error estimate  (see \cite[Lemma 5.4]{MK-2021})
\begin{equation} \label{estimate-0a}
\|E_{1-\alpha}(t_n)u_0-\bar U_h^n\|\leq c(h^2+\tau t_n^{\alpha/2-1})\|u_0\|_{L^{2p}(\Omega;\dot H^2)},
\end{equation} 
we deduce that 
\begin{eqnarray*}
\|e^n_1\|_{L^{2p}(\Omega;\dot H)}&\leq &c(h^2+\tau^{\min\{\alpha/2,\;\mu,\;\zeta\}}+\tau t_n^{\alpha/2-1} ).
\end{eqnarray*}
Given that  $\tau t_n^{\alpha/2-1}\leq c_T\tau^{\alpha/2}$, the proof is now complete.
\end{proof}
The next lemma will be used to derive an  estimate for $e^n_2$ in $L^{2p}(\Omega, \dot H)$-norm.
\begin{lemma}\label{sum_e_2}
Let $\tilde u_h^n$ and $U_h^n$ be the solutions of \eqref{DiscTldu} and \eqref{fully-1}, respectively, with $u_0\in L^{16p}(\Omega;\dot H^2)$. 
Assume $\|A^{(\beta-2)/2}\|_{{\cal L}_0^2}<\infty$ for some $\beta\in [1,3]$ and $\eta>0$. Then
\begin{eqnarray}\label{intErrorTilde}
\small
\tau\sum_{j=1}^n t_{n-j}^{\alpha-1}\|\nabla e^j_2\|^2_{L^{2p}(\Omega, \dot H)}&\leq &C h^2+C\tau^{\min\{\alpha/2,\;\mu,\; \zeta\}},   
\end{eqnarray}
where $\mu$ and $\zeta$ are given by \eqref{mu} and \eqref{zeta}, respectively.
\end{lemma}
\begin{proof}
From \eqref{DiscTldu} and \eqref{fully-1}, we see that
\begin{eqnarray}\label{errorTildeD}
 \dop  e^n_2+A_h^2e^n_2=A_hP_h(\phi(u_h(t_n))-\phi(U^{n}_h)),\quad\  e^0_2=0.
 \end{eqnarray}
Taking the inner product with $A^{-1}_he^n_2$, we get
\begin{eqnarray*}
 (\dop A^{-1/2}_h  e^n_2,A^{-1/2}_h e^n_2) +\|\nabla e^n_2\|^2&=&(\phi(u_h(t_n))-\phi(\tilde u_h^n),e^n_2)+(\phi(\tilde u_h^n)-\phi(U^{n}_h),e^n_2)\\
 & \leq &\frac{1}{2}\|\nabla e^n_2\|^2+\frac{1}{2}\|A^{-1/2}( \phi(u_h(t_n))-\phi(\tilde u_h^n)\|^2+\|e^n_2\|^2.
 \end{eqnarray*}
Using the fact $(\dop A^{-1/2}_he^n_2, A^{-1/2}_he^n_2)\geq \frac{1}{2}\dop\|A^{-1/2}_he^n_2\|^2 $ and $\|e^n_2\|^2\leq \|\nabla e^n_2\| \|e^n_2\|_{\dot H^{-1}}$, we obtain
\begin{eqnarray*}
  \frac{1}{2}\dop\|A^{-1/2}e^n_2\|^2 +\|\nabla e^n_2\|^2
  &\leq &\frac{1}{2}\|\nabla e^n_2\|^2+\frac{1}{2}\|A^{-1/2}( \phi(u_h(t_n))-\phi(\tilde u_h^n))\|^2\\
  & &+\frac{1}{3}\|\nabla e^n_2\|^2+\frac{3}{4}\|e^n_2\|_{\dot H^{-1}}^2,
 \end{eqnarray*}
and after simplication,
 $$  \dop\|A^{-1/2}e^n_2\|^2 +\|\nabla e^n_2\|^2
  \leq c\|A^{-1/2}( \phi(u_h(t_n))-\phi(\tilde u_h^n))\|^2+c\|e^n_2\|_{\dot H^{-1}}^2 . $$
Applying the discrete fractional integral operator $\partial_\tau^{-\alpha}$ to both sides, yields
 \begin{eqnarray*}
 \|A^{-1/2}e^n_2\|^2+\tau^{\alpha}\sum_{j=1}^n a_{n-j}^{(-\alpha)}\|\nabla e^j_2\|^2&\leq &c\tau^{\alpha}\sum_{j=1}^n a_{n-j}^{(-\alpha)}\|A^{-1/2}(\phi(u_h(t_j))-\phi(\tilde u_h^j))\|^2,
\end{eqnarray*}
and, using \eqref{C_3m},  
 \begin{eqnarray*}
 \|A^{-1/2}e^n_2\|^2+\tau^{\alpha}\sum_{j=1}^n a_{n-j}^{(-\alpha)}\|\nabla e^j_2\|^2&\leq & C\tau^{\alpha}\sum_{j=1}^n a_{n-j}^{(-\alpha)}\|e^j_1\|^2 \left(1+\|\tilde u_h^j\|^4_{\dot H^1}+\|u_h(t_j)\|_{\dot H^1}^4\right).
\end{eqnarray*}
Now, we have 
\begin{eqnarray*}
\|A^{-1/2}e^n_2\|^2_{L^{2p}(\Omega, H)}&+&\tau^{\alpha}\sum_{j=1}^n a_{n-j}^{(-\alpha)}\|\nabla e^j_2\|^2_{L^{2p}(\Omega, H)}\\
&\leq& C\tau^{\alpha}\sum_{j=1}^n a_{n-j}^{(-\alpha)}\|e^j_1\|^2_{L^{4p}(\Omega, \dot H)} 
\left(1+\|\tilde u_h^j\|^4_{L^{8p}(\Omega,\dot H^1) }+\|u_h(t_j)\|_{L^{8p}(\Omega,\dot H^1 )}^4\right)\\
&\leq& C\tau^{\alpha}\sum_{j=1}^n a_{n-j}^{(-\alpha)}\big(h^2+\tau^{\min\{\alpha/2,\;\mu,\; \zeta\}}\big)\\
&\leq& C h^2+C\tau^{\min\{\alpha/2,\;\mu,\; \zeta\}} .
\end{eqnarray*}
This implies that 
\begin{eqnarray}\label{intErrorTilde1}
\small
\tau^{\alpha}\sum_{j=1}^n a_{n-j}^{(-\alpha)}\|\nabla e^j_2\|^2_{L^{2p}(\Omega, \dot H)}&\leq &C h^2+C\tau^{\min\{\alpha/2,\;\mu,\; \zeta\}}.
\end{eqnarray}
Using Lemma \ref{t_nANDa_n}, we finally obtain
\begin{eqnarray*}
\small
\tau\sum_{j=1}^n t_{n-j}^{\alpha-1}\|\nabla e^j_2\|^2_{L^{2p}(\Omega, \dot H)}&\leq &C h^2+C\tau^{\min\{\alpha/2,\;\mu,\; \zeta\}},
\end{eqnarray*}
which concludes  the proof.
\end{proof}
Now, we are ready to prove an estimate for  $e^n_2$ in $L^{2p}(\Omega;\dot H)$-norm.
\begin{lemma}\label{e_2}
Let $\tilde u_h^n$ and $U_h^n$ be the solutions of \eqref{DiscTldu} and \eqref{fully-1}, respectively, with $u_0\in L^{16p}(\Omega;\dot H^2)$. 
Assume $\|A^{(\beta-2)/2}\|_{{\cal L}_0^2}<\infty$ for some $\beta\in [1,3]$ and $\eta>0$. Then
\begin{equation}\label{fully-1-e2}
\|e^n_2\|_{L^{2p}(\Omega;\dot H)}\leq C h^2+C\tau^{\min\{\alpha/2,\;\mu,\;\zeta\}},
\end{equation}
where $\mu$ and $\zeta$ are given by \eqref{mu} and \eqref{zeta}, respectively.
\end{lemma}
\begin{proof}
In view of  \eqref{errorTildeD}, 
$\displaystyle
e^n_2=\sum_{j=0}^{n-1}R_{n-j}\int_{t_j}^{t_{j+1}} A_hP_h(\phi(U_h^{j+1})-\phi( u_h(t_{j+1})))\,ds.
$
Then, 
\begin{eqnarray*}
\| e^n_2\|_{L^{2p}(\Omega, H)}& \leq & \sum_{j=0}^{n-1}\int_{t_j}^{t_{j+1}} \|R_{n-j}A_hP_h(\phi(U_h^{j+1})-\phi(\tilde u_h^{j+1}))\|_{L^{2p}(\Omega, H)}\,ds\\
  & & +\sum_{j=0}^{n-1}\int_{t_j}^{t_{j+1}} \|R_{n-j}A_hP_h(\phi(\tilde u_h^{j+1}))-\phi( u_h(t_{j+1})))\|_{L^{2p}(\Omega, H)}\,ds\\
  &=:&I+II.
\end{eqnarray*}
We estimate $I$ using  {\eqref{C_3}}  and \eqref{R_n2} as follows
\begin{eqnarray*}
I &\leq & C\tau  \sum_{j=0}^{n-1}t_{n-j+1}^{3\alpha/4-1}\|\nabla P(\phi(U_h^{j+1})-\phi(\tilde u_h^{j+1}))\|_{L^{2p}(\Omega, \dot H)} \\
 &\leq &C \tau\sum_{j=0}^{n-1}t_{n-j+1}^{3\alpha/4-1} \|\nabla e_2^{j+1}\|_{L^{4p}(\Omega, \dot H)} \left(1+{\|A_h U_h^{j+1} \|^2_{L^{8p}(\Omega,\dot H)}+\|A_h\tilde u_h^{j+1})\|^2_{L^{8p}(\Omega,\dot H ) }}\right)\\
 &\leq & C h^2+C\tau^{\min\{\alpha/2,\;\mu,\; \zeta\}},  
\end{eqnarray*}
where \eqref{intErrorTilde} is used in the last step. Next, we see that 
\begin{eqnarray*}
II & \leq & \tau\sum_{j=0}^{n-1}t_{n-j+1}^{\alpha/4-1}\|A^{-1/2}_hP_h(\phi(\tilde u_h^{j+1}))-\phi( u_h(t_{j+1})))\|_{L^{2p}(\Omega, \dot H)}\\
& \leq & C\tau \sum_{j=0}^{n-1}t_{n-j+1}^{\alpha/4-1}\| e_1^{j+1}\|_{L^{4p} (\Omega,\dot H)}\left(1+\|\tilde u_h^{j+1}\|^2_{L^{8p}(\Omega,\dot H^1)}+\|u_h(t_{j+1})\|^2_{L^{8p}(\Omega,\dot H^1 ) }\right)\\
& \leq & C (h^2+\tau^{\min\{\alpha/2,\;\mu,\; \zeta\}}). 
\end{eqnarray*}
The last inequality follows by using \eqref{fully-1-e1}.
\end{proof}
The main result for the fully-discrete scheme \eqref{fully-1} is now presented.
\begin{theorem}\label{thm:fully-1}
Let $u$ and $u_h^n$ be the solutions of \eqref{main} and \eqref{fully-1}, respectively, with
$u_0\in L^{16p}(\Omega;\dot H^2)$. 
Assume $\|A^{(\beta-2)/2}\|_{{\cal L}_0^2}<\infty$ for some $\beta\in [1,3]$ and $\eta>0$. Then
\begin{equation}\label{fully-1-a}
\|u(t_n)-U_h^n\|_{L^{2p}(\Omega;\dot H)}\leq c\left(h^{\min\{2,\beta\}-r}+
\tau^{\min\{\alpha/2,\mu,\zeta\}}\right),
\end{equation}
where $r$ is given by \eqref{rr} and \eqref{rr1}, and $\mu$ and $\zeta$ are given by \eqref{mu} and \eqref{zeta}, respectively.
\end{theorem}
\begin{proof}  The estimate follows by using the triangle inequality and the bounds obtained in Theorem \ref{thm:6.51} and in Lemmas \ref{e_1} and \ref{e_2}.
\end{proof}

\begin{remark}
For trace class noise, the temporal convergence rate is 
$O\left(\tau^{\min\{\alpha/2,\alpha+\gamma-1/2\}} \right)$ if $\gamma<1/2$ and $O(\tau^{\alpha/2})$ if $\gamma> 1/2$ and $\alpha+\gamma-\frac{1}{2}>1$. In the limiting case $\alpha \to 1^-$ and $\gamma\to 0^+$, an $O(\tau^{1/2-\epsilon})$ convergence rate is obtained, which coincides with that for the classical stochastic Cahn--Hilliard equation \cite{19}.
\end{remark}


\section{Numerical Results}\label{sec:NE}
We present one-dimensional numerical results in the unit interval $\mathcal{D}=(0,1)$ to support our analysis. We consider the stochastic Cahn-Hilliard equation 
\begin{equation} \label{num}
\cop u+\epsilon^2 A^2u+ AP\phi(u)=\partial_t^{-\gamma}\dot W(t),
\end{equation}
where $\epsilon>0$ represents the interface width parameter. We apply the numerical scheme \eqref{fully-2} and examine the theoretical estimates obtained in Theorem \ref{thm:fully-1}. 
\subsection{Implementation of the noise term}
First, we briefly discuss the implementation of the noise $W(t)$. 
We recall the expansion of  $W(t)$ given in \eqref{noise}:
$$ W(t)=\sum_{j=1}^\infty \gamma^{1/2}_j e_j \beta_j(t)  ,$$
with $\beta_j$, $j = 1, 2, . . . ,$ are i.i.d. Brownian motions, and $\gamma_j$ and $e_j$ are
the eigenvalues (arranged in non-decreasing order with counted multiplicity) and the corresponding eigenfunctions of $Q$, respectively. 
 We will consider only the case where the covariance operator $Q$ shares the same eigenfunctions with the operator $A$.
Recall that in one-dimension, the eigenvalues $\lambda_j$ and eigenfunctions $e_j(x)$ of $A$  are given by $j^2\pi^2$ and $\sqrt{2}\cos(j\pi x)$, $j\in\mathbb{N}$, respectively. We let $\gamma_j=j^{-m},\; m\in\mathbb{N}$, and it can be easily seen that 
\begin{equation}\label{m}
\|A^{(\beta-2)/2}\|_{\mathcal{L}_0^2}^2=\text{Tr}(A^{(\beta-2)}Q)\sim\sum_{j\in\mathbb{N}}j^{2(\beta-2)-m}.
\end{equation}
For the implementation of $W(t)$, we first notice the $L^2(\mathcal{D})$-projection $P_hW(t) \in \dot S_h$  given by 
$$ (P_hW(t),\chi)=\sum_{j=1}^\infty \gamma^{1/2}_j \beta_j(t) (e_j,\chi)\approx \sum_{j=1}^L \gamma^{1/2}_j \beta_j(t) (e_j,\chi),\quad \forall \chi \in V_h,$$
where the truncation number $L$ is appropriately selected. Since $\beta_j$ are i.i.d. Brownian motions, the increments $\Delta\beta^k_j$ satisfy 
$$\Delta\beta^k_j=\beta_j(t_k)- \beta_j(t_{k-1})\sim \sqrt{\tau}\mathcal{N}(0,1),\quad k=1,2,...,N,$$
where $\mathcal{N}(0,1)$ denotes the standard normal distribution.
Further, we approximate the term $P_h\dot W(t_k)$ by backward difference  
$$P_h\dot W(t_k)\approx \dfrac{P_hW(t_k)-P_hW(t_{k-1})}{\tau}\qquad \mbox{with } \; P_h \dot W(t_0)=0.$$
Then, applying the convolution quadrature \eqref{QC}, the term $\partial_t^{-\gamma} P_h\dot W(t_n)$ is approximated as follows:  
$$\partial_t^{-\gamma} P_h\dot W(t_n)\approx \tau^{\gamma}\sum_{k=1}^n a_{n-k}^{(-\gamma)}\sum_{j=1}^L \gamma_j^{1/2} e_j \frac{\Delta\beta^k_j}{\tau},$$
where $L$ is chosen to match to the FEM degree of freedom in order to to ensure  the desired convergence \cite{Yan-2005}.  

\subsection{Numerical tests and discussions}
For the numerical tests, we choose a uniform  time step $\tau=T/N$ and divide the the domain $\mathcal{D}$ into $M$ equal subintervals of length $h$. We then consider the stochastic  equation \eqref{num} with the following initial data:
\begin{itemize}
\item[ (a)]  $u_0(x)=0$ with $\epsilon=1$,  
\item[ (b)]  $u_0(x)=0.05\cos(2\pi x)\in \dot H^2(\mathcal{D})$,  $\epsilon=0.1$.   
\end{itemize}
The first case is presented to investigate the contribution of the noise without the effect of initial data and interface width parameter.  To check the convergence rates in space and time, we compute the $L^2(\Omega; \dot H)$-norm of the error $e^n=u(t_n)-U^n_h$ and evaluate the expected value over 100 sample paths.  The number within brackets in the last column of each table denotes the theoretical rate stated in Theorem \ref{thm:fully-1}. The condition $\eta>0$ is imposed in all numerical tests.

We first examine the temporal convergence. Since an exact solution is not available, we compute a reference solution with a much finer temporal mesh with $N = 2560$. 
To investigate the influence of the noise regularity (indicated by $m$) on the convergence rates, we present the numerical results in Tables \ref{table:1}-\ref{table:3} for $m = 0,1,2$, respectively, with  different values of $\gamma$ and $\alpha$. Note that, from \eqref{m}, the borderline for trace class noise is $m = 1$ corresponding to  \textcolor{black}{roughly $ \beta=2$}, while $m = 0$ corresponds a white noise with roughly  $\beta = 3/2$.

By Theorem \ref{thm:fully-1}, the error behaves like $O(\tau^{\min\{\alpha/2,\; \mu\}})$ for $m=0,1$.
The results presented in Tables 1 and 2 show that the empirical rates  closely align with the theoretical ones in case (a), while they are slightly higher than the anticipated rates 
in case (b). The results for the trace class noise with $m=2$  is presented in Table \ref{table:3} with  different values of $\gamma$ and $\alpha$. The results are very close to those in Table \ref{table:2}. Indeed, several numerical experiments showed that
the  noise regularity beyond trace class  affects very little the temporal convergence.

For the spatial convergence, we fix a very small time step size in order to neglect the effect of the temporal error, and compute a reference solution  on a very finer spatial mesh.
We conduct numerical tests with $\gamma =0.6$ in which case the number $r$  in \eqref{rr} and \eqref{rr1} is zero. Then, by Theorem \ref{thm:fully-1}, the error behaves like $O(h^{\min\{2,\beta \}})$.
Table \ref{table:4} displays the computed error for both cases (a) and (b)  with different values of $\alpha$ for $m=0,1,2$. As observed from the table, the computational  convergence rates are
in excellent agreement with the theoretical predictions  for $m=1$ and $m=2$. By contrast, the empirical rates are found to be higher than the theoretical ones when $m = 0$ (white noise) in both cases (a) and (b), indicating an interesting superconvergence phenomenon. This case needs further investigation to identify the cause of positive effect on the convergence rates.

\begin{table}[t]
{\footnotesize
\begin{center} 
\caption{Temporal convergence rates for cases (a) and (b) with $m=0$ for different values of $\alpha$ and  $\gamma$ at $T=0.01$.}
\label{table:1}
\begin{tabular}{|c|c|c|ccccc|c|}
\hline
  $\gamma$ &$\alpha$ & Case$\backslash N$  & 20 & 40 & 80 & 160&320   & Rate\\ 
\hline 
 & 0.5 & (a)  &  3.19e-2 & 2.62e-2 &2.05e-2 &  
1.57e-2 &1.24e-2  &$0.34 $ (0.238)\\
 &  & (b)  &  1.03e-1 &7.46e-2 &5.63e-2 &  
3.89e-2 &2.68e-2 &$0.49 $ (0.238)\\  
 \cline{2-9}
0.3 & 0.75 & (a)  &  2.91e-3 & 2.15e-3 &    
1.57e-3  &1.08e-3 & 7.09e-4 & 0.51  (0.456)\\ 
 &  & (b)  &  1.56e-2 &8.72e-3   &4.93e-3 &  
2.84e-3&1.62e-3  &0.82 (0.456)\\ 
 \hline
 &0.5 & (a)  & 3.27e-3 & 2.34e-3 &  
1.61e-3 &9.88e-4 & 6.62e-4 & 0.58 (0.438)\\
 &  & (b)  & 1.43e-1 &9.93e-2 &5.97e-2 &3.00e-2 &1.10e-2   
 &  0.93(0.438)\\
\cline{2-9}
0.5 &0.75 & (a)  & 5.39e-4 & 3.75e-4 &2.38e-4 &   
1.42e-4   &9.30e-5&0.63 (0.656)\\
 &  & (b)  & 7.87e-3  &4.16e-3 &2.17e-3 &1.08e-3 & 5.30e-4   &0.97 (0.656)\\
\hline
 &0.25 & (a)  &3.02e-4 &  2.28e-4 &  
1.67e-4   &1.12e-4&  7.90e-5&0.48 (0.519)\\
 &  & (b)  & 1.40e-1 & 1.22e-1&9.59e-2&6.46e-2 &3.76e-2 &0.47 (0.519)\\
\cline{2-9}
 0.8 &0.5 & (a)  &8.51e-5 & 6.21e-5 &4.30e-5&   
3.24e-5   & 2.40e-5 &0.46 (0.738)\\
 &  & (b)  &3.41e-2&1.40e-2 &6.42e-3&3.01e-3 &1.39e-3&1.15 (0.738)\\
\cline{2-9}
  &0.75 & (a)  &9.17e-5 &5.30e-5 & 3.07e-5 &1.55e-5 &9.75e-6&0.81(0.956)\\
 &  & (b)  &4.84e-3&2.52e-3 &1.27e-3 &6.20e-4 &2.91e-4 &  
1.01 (0.956)\\  
\hline
\end{tabular}
\end{center}
}
\end{table}


\begin{table}[t]
{\footnotesize
\begin{center} 
\caption{Temporal convergence rates for cases (a) and (b) with $m=1$ for different values of $\alpha$ and  $\gamma$ at $T=0.01$.}
\label{table:2}
\begin{tabular}{|c|c|c|ccccc|c|}
\hline
  $\gamma$ &$\alpha$ & Case$\backslash N$  & 20 & 40 & 80 & 160&320   & Rate\\  
\hline
 & 0.5 & (a)  &  3.14e-2 &2.58e-2 & 2.00e-2 &   
1.52e-2 &1.19e-2 & 0.35 (0.425)\\
 &  & (b)  &8.51e-2 &6.12e-2&4.55e-2 &2.89e-2 &1.99e-2&0.52 (0.425)\\  
 \cline{2-9}
0.3 & 0.75 & (a)  &  2.24e-3  &1.54e-3 &1.08e-3 &  
7.09e-4 &4.50e-4 &0.58 (0.550)\\
 &  & (b)  &1.10e-2 & 5.91e-3 &3.13e-3 &1.64e-3 &8.43e-4&   
0.93 (0.550)\\  
 \hline
 &0.5 & (a)  & 3.24e-3 & 2.30e-3 &1.56e-3&   
9.37e-4 &6.16e-4 &0.60 (0.500)\\
 &  & (b)  &1.28e-1 & 7.15e-2&2.69e-2&1.17e-2 &5.52e-3&   
1.13(0.500)\\
\cline{2-9}
 0.5&0.75 & (a)  & 5.01e-4 & 3.47e-4 &2.20e-4  & 
1.29e-4 & 8.50e-5& 0.64 (0.750)\\
 &  & (b)  &6.14e-3 &3.22e-3&1.66e-3 &8.13e-4 &3.96e-4 &  
0.99 (0.750)\\
\hline
 &0.25 & (a)  &3.02e-4  & 2.28e-4 & 
1.67e-4&1.12e-4 & 7.88e-5 & 0.48 (0.550)\\
 &  & (b)  &1.41e-1 &1.19e-1 &8.93e-2&5.86e-2 &3.29e-2 &  
0.53 (0.550)\\
\cline{2-9}
 0.8 &0.5 & (a)  &8.43e-5 &6.15e-5 &4.25e-5 &  
3.21e-5 &2.39e-5 &0.45 (0.800)\\
 &  & (b)  & 3.28e-2 &1.36e-2 &6.27e-3&2.94e-3 &  
1.36e-3 & 1.15 (0.800)\\
\cline{2-9}
  &0.75 & (a)  &9.15e-5 &5.27e-5 &3.04e-5 &  
1.51e-5 &9.49e-6&0.82 (1.000)\\
 &  & (b)  & 4.73e-3  &2.45e-3 &1.24e-3 & 6.04e-4 &  
2.84e-4& 1.01 (1.000)\\  
\hline
\end{tabular}
\end{center}
}
\end{table}

\begin{table}[t]
{\footnotesize
\begin{center} 
\caption{Temporal convergence rates for cases (a) and (b) with different values of $\alpha$ and  $\gamma$ at $T=0.01,\; m=2$.}
\label{table:3}
\begin{tabular}{|c|c|c|ccccc|c|}
\hline
  $\gamma$ &$\alpha$ & Case$\backslash N$  & 20 & 40 & 80 & 160&320   & Rate\\ 
\hline 
0.3 & 0.5 & (a)  & 3.12e-2&2.55e-2 &  
1.97e-2   &1.49e-2&1.16e-2&0.36 (0.300)\\ 
& & (b)  & 8.47e-2  &6.56e-2 &4.04e-2 & 2.38e-2 &
1.56e-2&0.61 (0.300)\\ 
 \cline{2-9}
 & 0.75 & (a)  & 1.86e-3 & 1.18e-3 &7.84e-4&   
4.94e-4&3.05e-4&    0.65 (0.550)\\
 &  & (b)  & 8.06e-3 &4.26e-3 &2.21e-3  &1.11e-3&   
5.50e-4 &0.97 (0.550)\\
 \hline 
 0.5 &0.5 & (a)  &3.22e-3 &2.28e-3 &1.54e-3 &9.11e-4&5.92e-4&0.61  (0.500)\\
  & & (b)  &8.37e-2 &4.52e-2 &2.20e-2 &1.09e-2 &4.69e-3& 1.04 (0.500)\\
\cline{2-9}
  &0.75 & (a)  &4.83e-4 &3.34e-4 &2.12e-4 &1.24e-4 &  
8.18e-5   &0.64 (0.750)\\
& & (b)  &5.31e-3 & 2.77e-3 &1.42e-3 &6.92e-04 &3.36e-4 & 0.99 (0.750)\\ 
\hline 
 &0.25 & (a)  &3.02e-4  &2.28e-4 &1.66e-4 &1.12e-4 &7.87e-5&0.48 (0.550) \\
& & (b) & 1.44e-1 &1.21e-1 & 8.56e-2  & 5.58e-2 &2.94e-2 & 0.57 (0.550)\\
\cline{2-9}
 0.8 &0.5 & (a)  &8.38e-5 &6.12e-5 &4.23e-5 & 3.19e-5 &2.38e-5  &0.45 (0.800)\\
   & & (b)  &3.22e-2 &1.34e-2 &6.20e-3 &2.91e-3 &  
1.35e-3  & 1.15 (0.800)\\
\cline{2-9}
  &0.75 & (a)  &9.14e-5 &5.26e-5 &3.02e-5 &1.49e-5 &9.36e-6 &0.82 (1.000)\\ 
    & & (b)  &4.68e-3 & 2.43e-3  & 1.23e-3   & 5.99e-4 & 2.81e-4 & 1.01 (1.000)\\ 
\hline
\end{tabular}
\end{center}
}
\end{table}


\begin{table}[t]
{\footnotesize
\begin{center} 
\caption{Spatial convergence rates for cases (a) and (b) with different values of $\alpha$ and $m$ when  $\gamma=0.6$ at  $T=0.01$.}
\label{table:4}
\begin{tabular}{|c|c|c|ccccc|c|}
\hline
  $m$ &$\alpha$ & Case$\backslash  h$  & 1/20 & 1/40 & 1/80 & 1/160&1/320   & Rate\\ 
\hline 
 & 0.3 & (a) &6.12e-5 & 1.53e-5 &3.80e-6 & 9.14e-7&    
2.00e-7 &2.09 (1.50) \\  
  & & (b) & 1.05e-2  & 2.68e-3 & 6.65e-4 & 1.59e-4 & 3.20e-5   & 2.13 (1.50)\\  
\cline{2-9}
 m=0 & 0.5 & (a)  &5.04e-5  & 1.26e-5& 3.11e-6& 7.39e-7& 1.45e-7 &2.15 (1.50) \\ 
& & (b)  & 7.19e-3  &1.83e-3 & 4.57e-4 & 1.11e-4 & 2.50e-5 & 2.06 (1.50)\\ 
 \cline{2-9}
 & 0.75 & (a)  & 3.54e-5 & 8.85e-6 &2.19e-6 &5.19e-7 &1.02e-7 & 2.15 (1.50)\\
 &  & (b)  & 1.28e-3  & 3.26e-4 & 8.14e-5 & 1.98e-5 & 4.37e-6 & 2.07 (1.50)\\
 \hline 
 &0.3 & (a)& 6.02e-5 &1.51e-5 & 3.73e-6 &8.99e-7 &   
1.97e-7 & 2.08 (2.00)\\
 & & (b)& 1.08e-2 & 2.71e-3 & 6.73e-4 & 1.60e-4 & 3.23e-5 & 2.13 (2.00)\\
\cline{2-9}
 m=1 &0.5 & (a)  &4.92e-5 &1.23e-5 &3.04e-6 & 7.21e-7 & 1.41e-7 & 2.15 (2.00)\\
  & & (b)  &7.20e-3 & 1.83e-3 & 4.57e-4  &1.15e-4 &   
2.51e-5 & 2.06 (2.00)\\
\cline{2-9}
  &0.75 & (a)  &3.38e-5  & 8.45e-6 &2.09e-6  & 4.95e-7 & 9.73e-8 & 2.15 (2.00)\\
& & (b)  & 8.60e-4 & 2.17e-4 & 5.42e-5 & 1.34e-5 &3.26e-6 & 2.02 (2.00)\\ 
\hline 
 &0.3 & (a)& 5.97e-5 & 1.49e-5 & 3.70e-6 &8.92e-7&   
1.96e-7 & 2.08 (2.00)\\
 & & (b)&3.18e-2 & 7.32e-3 & 1.56e-3 & 3.50e-4  & 5.95e-5 & 2.31  (2.00)\\
\cline{2-9}
 m=2 &0.5 & (a)  &4.86e-5 & 1.21e-5 & 3.00e-6 & 7.12e-7 & 1.39e-7 & 2.15  (2.00)\\
  & & (b)  & 7.22e-3 & 1.84e-3  & 4.58e-4  &  1.12e-4 & 2.52e-5 & 2.06 (2.00)\\
\cline{2-9}
  &0.75 & (a)  &3.32e-5 & 8.28e-6 & 2.04e-6 & 4.84e-7 & 9.53e-8 & 2.15  (2.00)\\
& & (b)  & 6.88e-4 & 1.72e-4 & 4.32e-5 & 1.09e-5 & 2.84e-6 & 1.97  (2.00)\\ 
\hline 
\end{tabular}
\end{center}
}
\end{table}


\begin{center}
{\bf Appendix A. Proof of Lemma \ref{lem:a3}}
\end{center}

\begin{proof}
To prove the lemma, we proceed as in \cite[Lemma 5.2]{MK-2023}. We first split the integrand as 
\begin{eqnarray*}
E_{0,h}(t_n-t)A_h\phi_h(u_h(t))-R_{n-j}A_h\phi_h(u_h(t_{j+1})))& =&
(E_{0,h}(t_n-t)-E_{0,h}(t_n-t_j))A_h\phi_h(u_h(t))\\
& &+E_{0,h}(t_n-t_j)A_h(\phi_h(u_h(t))-\phi_h(u_h(t_{j+1})))\\
& &+(E_{0,h}(t_n-t_j)-
R_{n-j})A_h\phi_h(u_h(t_{j+1}))\\
 &=: &\sum_{k=1}^3I_k.
\end{eqnarray*}
Then, after integration, 

\begin{equation}
\begin{split}
\sum_{j=0}^{n-1}\int_{t_j}^{t_{j+1}}& \|I_{1}\|_{L^{2p}(\Omega; \dot H)}\;dt \leq   \sum_{j=0}^{n-2}\int_{t_j}^{t_{j+1}}\|I_{1}\|_{L^{2p}(\Omega; H)}\; dt+\int_{t_{n-1}}^{t_{n}}\|I_{1}\|_{L^{2p}(\Omega; \dot H)}\; dt\\
  & \leq \sum_{j=0}^{n-2} \int_{t_j}^{t_{j+1}}\int_{t_j}^{t}\|A_hE_{0,h}'(t_n-s)\phi_h(u_h(t))\|_{L^{2p}(\Omega; \dot H)}\; ds\; dt \\  
  & \;\; +\int_{t_{n-1}}^{t_{n}}\|A_hE_{0,h}(t_n-t)\phi(u_h(t))\|_{L^{2p}(\Omega; \dot H)}\; dt\\
  & \;\; + \int_{t_{n-1}}^{t_{n}}\|A_hE_{0,h}(\tau)P_h\phi(u_h(t))\|_{L^{2p}(\Omega; \dot H)}\; dt\\
  & \leq c\sum_{j=0}^{n-2} \int_{t_j}^{t_{j+1}}\int_{t_j}^{t}(t_n-s)^{\alpha/2-2}\; ds\; dt  +c \int_{t_{n-1}}^{t_{n}}(t_n-t)^{\alpha/2-1}\; dt+c \int_{t_{n-1}}^{t_{n}}\tau^{\alpha/2-1}\; dt\\
   & \leq c\sum_{j=0}^{n-2} \int_{t_j}^{t_{j+1}} \tau (t_n-t)^{\alpha/2-2} \; dt  +c \int_{t_{n-1}}^{t_{n}}(t_n-t)^{\alpha/2-1}\; dt+c \int_{t_{n-1}}^{t_{n}}\tau^{\alpha/2-1}\; dt\\
    & \leq c \int_{0}^{t_{n-1}} \tau (t_n-t)^{\alpha/2-2} \; dt  +c \tau^{\alpha/2}
    \leq  c \tau^{\alpha/2}.
\end{split}   
\end{equation}
Now, using \eqref{Ah} and Theorem \ref{local} with $q=0$, we get 
\begin{eqnarray*}
\sum_{j=0}^{n-1}\int_{t_j}^{t_{j+1}}\|I_2\|_{L^{2p}(\Omega; \dot H)}\; dt &\leq & \sum_{j=0}^{n-1} \int_{t_j}^{t_{j+1}}\|A_h^{3/2}E_{0,h}(t_n-t_j)A_h^{-1/2}(\phi_h(u_h(t))-\phi_h(u_h(t_{j+1})))\|_{L^{2p}(\Omega; \dot H)}\;  dt   \\
   & \leq &c\sum_{j=0}^{n-1}(t_n-t_j)^{\alpha/4-1} \int_{t_j}^{t_{j+1}}(t_{j+1}-t)^{\min\{\alpha/2,\; \theta-1/2\}}\; dt   \\
    & \leq &c\sum_{j=1}^{n}t_j^{\alpha/4-1} \tau^{\min\{\alpha/2,\; \eta+\alpha/4,\;{\alpha+\gamma-1/2}\}+1} 
    \leq c t_n^{\alpha/4} \tau^{\min\{\alpha/2,\; \eta+\alpha/4,\; {\alpha+\gamma-1/2}\}}.
\end{eqnarray*}

For $I_3$, we use the estimate $\|A_h(E_{0,h}(t_n)-R_n)P_h\|\leq c\tau t_{n+1}^{\alpha/2-2}$, see Lemma \ref{PQ},  to get 
 \begin{eqnarray*}
\sum_{j=0}^{n-1}\int_{t_j}^{t_{j+1}}\|I_{3}\|_{L^{2p}(\Omega; H)}\; dt &\leq & \sum_{j=0}^{n-1} \int_{t_j}^{t_{j+1}}\|A_h( E_{0,h}(t_n-t_j)-
R_{n-j}){\phi(u_h(t_{j+1}}))\|_{L^{2p}(\Omega; \dot H)}\; dt   \\
  & \leq &c\sum_{j=0}^{n-1} \int_{t_j}^{t_{j+1}}t_{n-j+1}^{\alpha/2-2}\tau \;dt  
  \leq c\tau^2\sum_{j=0}^{n-1} t_{n-j+1}^{\alpha/2-2} \\
  & = &c\tau^2\sum_{j=1}^{n} t_{j+1}^{\alpha/2-2} 
  \leq c\tau^{\alpha/2}.
\end{eqnarray*} 
Gathering all the previous estimates completes the proof.
\end{proof}
\begin{center}
	{\bf Appendix B. Proof of Lemma \ref{lem:a2}}
\end{center}
		
\begin{proof}
The proof closely follows the one presented \cite[Theorem 4.1]{BYZ-2019}. For the reader's convenience, we will provide the detailed proof. We first bound the left hand side (LHS) in \eqref{mis} as follows
\begin{eqnarray*}
LHS&\leq&\left( \sum_{j=0}^{n-1}\int_{t_j}^{t_{j+1}}\|A^{(2-\beta)/2}( E_{\gamma,h}(t_n-t)- E_{\gamma,h}(t_n-t_j))A^{(\beta-2)/2}\|_{\mathcal{L}_2^0}^2dt\right)^{1/2}\\
&& +\left( \sum_{j=0}^{n-1}\int_{t_j}^{t_{j+1}}\|A^{(2-\beta)/2}(  E_{\gamma,h}(t_n-t_j)-Q_{n-j})A^{(\beta-2)/2}\|_{\mathcal{L}_2^0}^2dt\right)^{1/2}:=I_1^{1/2}+I_2^{1/2}.
\end{eqnarray*}
To estimate $I_1$, we first observe that a straightforward interpolation between $\beta=1$ and $\beta=2$ enables replacing  $A$ by $A_h$, and thus,
\begin{eqnarray*}
I_1&\leq&\sum_{j=0}^{n-2}\int_{t_j}^{t_{j+1}}\|A_h^{(2-\beta)/2}( E_{\gamma,h}(t_n-t)- E_{\gamma,h}(t_n-t_j))A_h^{(\beta-2)/2}\|_{\mathcal{L}_2^0}^2dt\\
&& +\int_{t_{n-1}}^{t_n}\|A_h^{(2-\beta)/2}( E_{\gamma,h}(t_n-t)- E_{\gamma,h}(\tau))A_h^{(\beta-2)/2}\|_{\mathcal{L}_2^0}^2dt:=I_{1,1}+I_{1,2}.
\end{eqnarray*}
To estimate $I_{1,1}$, we apply H\"older inequality and exploit  the smoothing property of $  E_{\gamma,h}'(t)$ so that
\begin{eqnarray*}
I_{1,1}&\leq&\sum_{j=0}^{n-2}\int_{t_j}^{t_{j+1}}\left\|\int_{t_j}^s A_h^{(2-\beta)/2} E_{\gamma,h}'(t_n-t)A_h^{(\beta-2)/2}\,dt\right\|_{\mathcal{L}_2^0}^2ds\\
&\leq&\sum_{j=0}^{n-2} \int_{t_j}^{t_{j+1}}\tau\int_{t_j}^s \|A_h^{(2-\beta)/2}  E_{\gamma,h}'(t_n-t)A_h^{(\beta-2)/2}\|_{\mathcal{L}_2^0}^2\,dtds\\
&\leq&c\tau^2 \|A_h^{(\beta-2)/2}\|_{\mathcal{L}_2^0}^2 \int_{\tau}^{t_n}\|A_h^{(2-\beta)/2} E_{\gamma,h}'(t)\|^2\,dt\leq 
c\tau^2\|A_h^{(\beta-2)/2}\|_{\mathcal{L}_2^0}^2  \int_{\tau}^{t_n}t^{2(\alpha(2+\beta)/4+\gamma-2)}\,dt. \\
\end{eqnarray*}
By noting that $2(\alpha(2+\beta)/4+\gamma-2)=2(\eta+\frac{\alpha}{4}-1)-1$, we readily obtain
$$
I_{1,1}\leq c\left\{\begin{array}{ll} \tau^{2(\eta+\frac{\alpha}{4})}, & \eta+\frac{\alpha}{4} <1,\\\tau^2\ell_n, &\eta+\frac{\alpha}{4}=1,\\\tau^2t_n^{2(\eta+\frac{\alpha}{4}-1)},&\eta+\frac{\alpha}{4}>1,\end{array}\right.
$$
where $\ell_n=\ln(\frac{t_n}{\tau})$. The estimate for $I_{1,2}$ is established by applying the triangle and Lemma \ref{eE}: 
\begin{eqnarray*}
I_{1,2}&\leq&c\int_0^\tau\|A_h^{(2-\beta)/2}  E_{\gamma,h}(t)A_h^{(\beta-2)/4}\|_{\mathcal{L}_2^0}^2 dt+c\int_0^\tau\|A_h^{(2-\beta)/2} E_{\gamma,h}(\tau)A_h^{(\beta-2)/4}\|_{\mathcal{L}_2^0}^2 dt\\
&\leq&c\|A_h^{(\beta-2)/2}\|_{\mathcal{L}_2^0}^2 \int_0^\tau t^{2(\alpha(2+\beta)/4+\gamma-1)}\,dt+c\|A_h^{(\beta-2)/2}\|_{\mathcal{L}_2^0}^2 \tau^{2(\alpha(2+\beta)/4+\gamma-1)+1}\\
&\leq &c\|A_h^{(\beta-2)/2}\|_{\mathcal{L}_2^0}^2\tau^{2(\eta+\alpha/4)}.
\end{eqnarray*}
For the last term $I_2$, we  use Lemma \ref{Q_n} to get
\begin{eqnarray*}
I_2&\leq&c \sum_{j=0}^{n-1}\int_{t_j}^{t_{j+1}}\|A^{(2-\beta)/2}( E_{\gamma,h}(t_n-t_j)-Q_{n-j})A^{(\beta-2)/2}\|_{\mathcal{L}_2^0}^2dt\\
&\leq&c \tau^3 \sum_{j=0}^{n-1}(t_{n}-t_j)^{2(\alpha(2+\beta)/4+\gamma)-2}\leq c\|A_h^{(\beta-2)/2}\|_{\mathcal{L}_2^0}^2
\left\{\begin{array}{ll} \tau^{2(\eta+\alpha/4)}, & \eta +\alpha/4<1,\\\tau^2\ell_n, &\eta+\alpha/4=1,\\ \tau^2t_n^{2(\eta+\alpha/4-1)},&\eta+\alpha/4>1.\end{array}\right.
\end{eqnarray*}
Combining the preceding estimates  completes the proof of the lemma.
\end{proof}

\noindent {\bf Declaration of competing interest}

The authors declare that they have no known competing financial interests or personal relationships that could have
appeared to influence the work reported in this paper.

\noindent {\bf Data availability} 

No data was used for the research described in the article.

\end{document}